\pdfoutput=1
\RequirePackage{ifpdf}
\ifpdf 
\documentclass[pdftex]{sigma}
\else
\documentclass{sigma}
\fi

\numberwithin{equation}{section}

\newtheorem{Theorem}{Theorem}[section]
\newtheorem*{Theorem*}{Theorem}

\newtheorem{Lemma}[Theorem]{Lemma}

 { \theoremstyle{definition}
\newtheorem{Definition}[Theorem]{Definition}

\newtheorem{Remark}[Theorem]{Remark} }

\newtheorem*{theoremA1}{Theorem A1}
\newtheorem*{theoremA2}{Theorem A2}
\newtheorem*{theoremA3}{Theorem A3}
\newtheorem*{theoremA4}{Theorem A4}

\newtheorem*{theoremG1}{Theorem G1}
\newtheorem*{theoremG2}{Theorem G2}
\newtheorem*{theoremG3}{Theorem G3}
\newtheorem*{theoremG4}{Theorem G4}
\begin{document}

\allowdisplaybreaks

\newcommand{\arXivNumber}{2406.10947}

\renewcommand{\PaperNumber}{107}

\FirstPageHeading

\ShortArticleName{The Algebraic and Geometric Classification of Compatible Pre-Lie Algebras}

\ArticleName{The Algebraic and Geometric Classification\\ of Compatible Pre-Lie Algebras}

\Author{Hani ABDELWAHAB~$^{\rm a}$, Ivan KAYGORODOV~$^{\rm b}$ and Abdenacer MAKHLOUF~$^{\rm c}$}

\AuthorNameForHeading{H.~Abdelwahab, I.~Kaygorodov and A.~Makhlouf}

\Address{$^{\rm a)}$~Department of Mathematics, Faculty of Science, Mansoura University, Mansoura, Egypt}
\EmailD{\href{mailto:haniamar1985@gmail.com}{haniamar1985@gmail.com}}

\Address{$^{\rm b)}$~CMA-UBI, University of Beira Interior, Covilh\~{a}, Portugal}
\EmailD{\href{mailto:kaygorodov.ivan@gmail.com}{kaygorodov.ivan@gmail.com}}

\Address{$^{\rm c)}$~IRIMAS - D\'epartement de Math\'ematiques, University of Haute Alsace, Mulhouse, France}
\EmailD{\href{mailto:abdenacer.makhlouf@uha.fr}{abdenacer.makhlouf@uha.fr}}

\ArticleDates{Received August 22, 2024, in final form November 18, 2024; Published online November 28, 2024}

\Abstract{In this paper, we develop a method to obtain the algebraic classification of compatible pre-Lie algebras from the classification of pre-Lie algebras of the same dimension. We use this method to obtain the algebraic classification of complex 2-dimensional compatible pre-Lie algebras. As a byproduct, we obtain the classification of complex 2-dimensional compatible commutative associative, compatible associative and compatible Novikov algebras. In addition, we consider the geometric classification of varieties of cited algebras, that is the description of its irreducible components.}


\Classification{17A30; 17D25; 14L30}

\section{Introduction}

The algebraic classification (up to isomorphism) of algebras of
small dimensions from a certain variety defined by a family of polynomial identities is a 
classic problem in the theory of non-associative algebras.
Another interesting approach to studying algebras of a fixed dimension is to study them from a geometric point of view (that is, to study the degenerations and deformations of these algebras). The results in which the complete information about degenerations of a~certain variety is obtained are generally referred to as the geometric classification of the algebras of these varieties. There are many results related to the algebraic and geometric classification of
Jordan, Lie, Leibniz, Zinbiel, and other algebras
(see \cite{abcf, fkkv, afm, BB09,BB14, EM22, BT22, GRH, KKL, KV, LLL} and references in \cite{k23,MS,l24}). The geometric classification of algebras from a certain variety is based on the notion of degeneration, that is a ``dual'' notion to deformation \cite{b,D23,EM22,HSZ,lsb20,c}.

Pre-Lie algebras, also known as right symmetric algebras,
appeared in some papers by Gerstenhaber, Koszul, and Vinberg in the 1960s.
It is a generalization of associative algebras and the most popular non-associative subvariety of Lie-admissible algebras.
They have various applications in geometry and physics \cite{B06};
recently some connections between them and trees, braces and $F$-manifold algebras were established in \cite{CL,lsb20, SS24,S22}.
Novikov algebras are pre-Lie algebras with an additional identity.
They were introduced in papers by
Gel'fand and Dorfman, Balinskii and Novikov (about Novikov algebras, see \cite{DIU} and references therein). Let $\Omega$ be a~variety of algebras.
 We say that an algebra $(\bf A, \cdot, \ast)$ is a compatible $\Omega$-algebra,
 if and only if~$(\bf A, \cdot)$, $(\bf A, \ast)$
and $(\bf A, \cdot + \ast)$ are $\Omega$-algebras.
Compatible Lie algebras are considered in the study of the classical Yang--Baxter equation \cite{GS06},
integrable equations of the principal chiral model type \cite{GS02},
elliptic theta functions \cite{OS04}, and other areas.
The study of non-Lie compatible algebras is also very popular.
So, cohomology and deformations were studied for
compatible Lie algebras \cite{lsb23},
compatible associative algebras \cite{CDM},
compatible ${\rm Hom}$-Lie algebras \cite{D23},
compatible~$L_{\infty}$-algebras \cite{D22},
compatible $3$-Lie algebras \cite{HSZ},
compatible dendriform algebras \cite{DSQ}, and so on.
Free compatible algebras were studied in the associative algebra case in \cite{D09} and
in the Lie algebra case in \cite{GR,l10}.
Compatible associative structures on matrix algebras were studied in~\cite{OS05,OS06,OS08}.
General compatible structures have been studied from an operadic point of view in~\cite{S08,zgg}.
A generalization of compatible algebras was considered in \cite{Khr}.
At this moment there is only one paper about the classification of small-dimensional compatible algebras.
Namely, all~$4$-dimensional nilpotent compatible Lie algebras were classified in \cite{LLL}.

The main goal of the present paper is to obtain the algebraic and geometric description of the variety of complex $2$-dimensional compatible pre-Lie algebras. To do so, we first determine all such $2$-dimensional algebra structures, up to isomorphism (what we call the algebraic classification), and then proceed to determine the geometric properties of the corresponding variety, namely its dimension and description of the irreducible components (the geometric classification).
As some corollaries, we have the algebraic and geometric classification of complex $2$-dimensional
compatible commutative associative,
compatible associative
and compatible Novikov algebras.

Our main results regarding the algebraic classification are summarized below.

\begin{theoremA1}
There are infinitely many isomorphism classes of complex
$2$-dimensional compatible pre-Lie algebras, described explicitly in
 Theorem {\rm\ref{compre2}}
 in terms of
$6$ three-parameter families,
$14$ two-parameter families,
$13$ one-parameter families,
and
$8$ additional isomorphism classes.
\end{theoremA1}

\begin{theoremA2}
There are infinitely many isomorphism classes of complex
$2$-dimensional compatible commutative associative algebras, described explicitly in
 Theorem {\rm\ref{comasscom2}}
 in terms of
$1$ three-parameter family,
$3$ two-parameter families,
$8$ one-parameter families,
and
$6$ additional isomorphism classes.
\end{theoremA2}

\begin{theoremA3}
There are infinitely many isomorphism classes of complex
$2$-dimensional compatible associative algebras, described explicitly in
 Theorem {\rm\ref{comass2}}
 in terms of
$1$ three-parameter family,
$3$ two-parameter families,
$10$ one-parameter families,
and
$10$ additional isomorphism classes.
\end{theoremA3}

\begin{theoremA4}
There are infinitely many isomorphism classes of complex
$2$-dimensional compatible Novikov algebras, described explicitly in
 Theorem {\rm\ref{comnov2}}
 in terms of
$3$ three-parameter families,
$7$ two-parameter families,
$11$ one-parameter families,
and
$7$ additional isomorphism classes.
\end{theoremA4}

The geometric part of our work aims to generalize previously obtained results about the geometric classification of $2$-dimensional Novikov \cite{BB14} and pre-Lie \cite{BB09} algebras. Our main results regarding the geometric classification are summarized below.

\begin{theoremG1}
The variety of complex $2$-dimensional compatible pre-Lie algebras has dimension~$7$.
It is defined by $2$ rigid algebras,
$1$ one-parametric family of algebras,
$7$ two-parametric families of algebras,
and $4$ three-parametric families of algebras and can be described as the closure of the union of $\mathrm{GL}_2(\mathbb{C})$-orbits of the algebras given in Theorem {\rm\ref{geo4}}.
\end{theoremG1}

\begin{theoremG2}
The variety of complex $2$-dimensional compatible commutative associative algebras has dimension $7$.
It is defined by
$1$ two-parametric family of algebras
and $1$ three-parametric family of algebras and can be described as the closure of the union of $\mathrm{GL}_2(\mathbb{C})$-orbits of the algebras given in Theorem {\rm\ref{geo2}}.
\end{theoremG2}

\begin{theoremG3}
The variety of complex $2$-dimensional compatible associative algebras has dimension $7$.
It is defined by $2$ rigid algebras,
$2$ two-parametric families of algebras,
and $1$ three-parametric family of algebras and can be described as the closure of the union of $\mathrm{GL}_2(\mathbb{C})$-orbits of the algebras given in Theorem {\rm\ref{geo3}}.
\end{theoremG3}

 \begin{theoremG4}
The variety of complex $2$-dimensional compatible Novikov algebras has dimension~$7$.
It is defined by $1$ rigid algebra,
$3$ two-parametric families of algebras,
and $1$ three-parametric family of algebras and can be described as the closure of the union of $\mathrm{GL}_2(\mathbb{C})$-orbits of the algebras given in Theorem {\rm\ref{geo1}}.
\end{theoremG4}

\section{The algebraic classification of compatible algebras}

All the algebras below will be over $\mathbb C$ and all the linear maps will be $\mathbb C$-linear.
For simplicity, every time we write the multiplication table of an algebra
the products of basic elements whose values are zero are omitted.

\subsection{The algebraic classification of algebras}
In this paper, we work with compatible pre-Lie algebras with two multiplications. Let us review the method we will use to obtain the algebraic classification for the variety of compatible pre-Lie algebras
(the present method, in the case of Poisson algebras, is given with more details in~\cite{fkkv,afm}).

\begin{Definition}
An algebra is called a pre-Lie algebra if it satisfies the identity
\[
( x\cdot y) \cdot z-x\cdot ( y\cdot
z)=( y\cdot x) \cdot z -y\cdot ( x\cdot z).
\]
\end{Definition}

\begin{Definition}
A compatible pre-Lie algebra is a vector space ${\bf A}$ equipped with
two multiplications: $\cdot$ and another multiplication $\ast$
such that $({\bf A}, \cdot)$, $({\bf A}, \ast)$ and
$({\bf A}, \cdot+ \ast)$ are pre-Lie algebras.
These two operations are required to satisfy the following identities:
\begin{gather*}
( x\cdot y) \cdot z-x\cdot ( y\cdot
z)=( y\cdot x) \cdot z -y\cdot ( x\cdot z),\\
( x\ast y) \ast z-x\ast ( y\ast
z)=( y\ast x) \ast z -y\ast ( x\ast z),\\
(x\ast y)\cdot z-x\ast (y\cdot z)+(x\cdot y)\ast z-x\cdot (y\ast z)\\
\qquad =
(y\ast x)\cdot z-y\ast (x\cdot z)+(y\cdot x)\ast z-y\cdot (x\ast z).
 \end{gather*}
The main examples of compatible pre-Lie algebras are the following:
compatible commutative associative,
compatible associative
and compatible Novikov algebras.
\end{Definition}

\begin{Definition}
Let $( {\bf A},\cdot ) $ be a pre-Lie
algebra. Define ${\rm Z}^{2}( {\bf A},{\bf A}) $ to be the
set of all bilinear maps
$\theta \colon{\bf A}\times {\bf A}\longrightarrow {\bf A}$ such that
\begin{gather*}
\theta(\theta( x, y), z)-\theta(x, \theta ( y, z))=\theta(\theta( y, x) , z) - \theta(y, \theta ( x, z)),\\
\theta(x, y)\cdot z-\theta(x, y\cdot z)
+\theta(x\cdot y, z) -x\cdot \theta(y, z)\\
\qquad =
\theta(y, x)\cdot z-\theta(y, x\cdot z)+
\theta(y\cdot x, z)-y\cdot \theta(x, z) .
 \end{gather*}
Then ${\rm Z}^{2}({\bf A},{\bf A})\neq {\varnothing }$ since $\theta=0\in {\rm Z}^{2}({\bf A},{\bf A})$.
\end{Definition}

Now, for $\theta \in {\rm Z}^{2}( {\bf A},{\bf A})$, let us define a multiplication
$\bullet_{\theta }$ on ${\bf A}$ by $x\bullet_{\theta}y=\theta ( x,y) $ for all~${x,y \in {\bf A}}$.
Then $( {\bf A},\cdot ,\bullet_{\theta }) $ is a compatible pre-Lie
algebra.
Conversely, if $( {\bf A},\cdot ,\bullet ) $ is a compatible pre-Lie algebra,
then there exists $\theta \in {\rm Z}^{2}(
{\bf A},{\bf A}) $ such that $(
{\bf A},\cdot ,\bullet_{\theta }) \cong( {\bf A},\cdot ,\bullet
)$. To see this, consider the bilinear map $\theta \colon{\bf A}\times {\bf A}\longrightarrow {\bf A}$ defined by $
\theta ( x,y) = x\bullet y$ for all $x,y$ in ${\bf A}$. Then $\theta \in {\rm Z}^{2}(
{\bf A},{\bf A}) $ and $(
{\bf A},\cdot ,\bullet_\theta ) =( {\bf A},\cdot ,\bullet
)$.

Let $( {\bf A},\cdot ) $ be a pre-Lie
algebra. The automorphism group ${\rm Aut}( {\bf A}) $
of ${\bf A}$ acts on ${\rm Z}^{2}( {\bf A},{\bf A}) $~by
\[
(\theta *\phi) ( x,y) =\phi ^{-1}( \theta ( \phi (
x) ,\phi ( y) ) )
\]
for $\phi \in {\rm Aut}
( {\bf A}) $ and $\theta \in {\rm Z}^{2}( {\bf A}, {\bf A}) $.

\begin{Lemma}
\label{isom}Let $( {\bf A},\cdot ) $ be a pre-Lie algebra
and $\theta ,\vartheta \in {\rm Z}^{2}( {\bf A},{\bf A}) $.
Then $( {\bf A},\cdot ,\bullet_{\theta }) $ and $%
( {\bf A},\cdot ,\bullet_{\vartheta }) $ are
isomorphic if and only if there is a linear map $\phi \in {\rm Aut}( {\bf A}
)$ such that $\theta *\phi =\vartheta $.
\end{Lemma}

\begin{proof}
If $\theta * \phi =\vartheta $, then $\phi \colon( {\bf A},\cdot ,\bullet_{\vartheta }) \longrightarrow $ $( {\bf A},\cdot
,\bullet_{\theta }) $ is an isomorphism since
$\phi
( \vartheta ( x,y) ) =\theta ( \phi (
x) ,\phi ( y) ) $. On the other hand, if $\phi
\colon( {\bf A},\cdot ,\bullet_{\vartheta })
\longrightarrow $ $( {\bf A},\cdot ,\bullet _{\theta
}) $ is an isomorphism of compatible pre-Lie algebras, then $\phi \in {\rm Aut}%
( {\bf A}) $ and $\phi ( x\bullet _{\vartheta
}y ) = \phi ( x) \bullet _{\theta } \phi ( y)
 $. Hence
 \[
 \vartheta ( x,y) =\phi ^{-1}( \theta
( \phi ( x) ,\phi ( y) ) ) =(\theta *
\phi )( x,y),
 \]
 and therefore $\theta * \phi=\vartheta $.
\end{proof}

Consequently, we have a procedure to classify the compatible pre-Lie algebras with the given
associated pre-Lie algebra $( {\bf A},\cdot )
$. It consists of three steps:
\begin{enumerate}\itemsep=0pt
\item[(1)] compute ${\rm Z}^{2}( {\bf A},{\bf A}) $,
\item[(2)] find the orbits of ${\rm Aut}( {\bf A}) $ on $%
{\rm Z}^{2}( {\bf A},{\bf A}) $,
\item[(3)] choose a representative $\theta$ from each orbit and then construct the compatible pre-Lie algebra $( {\bf A},\cdot,\bullet _{\theta }) $.
\end{enumerate}

\subsection{2-dimensional pre-Lie algebras}

\begin{Lemma}
 Let $\mathcal{C}$ be a nonzero $2$-dimensional pre-Lie algebra.
Then $\mathcal{C}$ is isomorphic to one and only one of the following
algebras:
\begin{gather*}
 \mathcal{C}_{01} \colon\ e_{1} \cdot e_{1} = e_{1} + e_{2},\quad e_{2} \cdot e_{1} =e_{2},\\
 \mathcal{C}_{02} \colon\ e_{1}\cdot e_{1} = e_{1} + e_{2},\quad e_{1}\cdot e_{2} =e_{2}, \\
 \mathcal{C}_{03} \colon\ e_{1}\cdot e_{1} =e_{2}, \\
 \mathcal{C}_{04} \colon\ e_{2} \cdot e_{1} =e_{1}, \\
 \mathcal{C}_{05}^{\alpha} \colon\ e_{1}\cdot e_{1} =e_{1}, \quad e_{1}\cdot e_{2} = \alpha e_{2}, \\
 \mathcal{C}_{06}^{\alpha} \colon\ e_{1} \cdot e_{1} =e_{1}, \quad e_{1} \cdot e_{2} = \alpha e_{2},\quad e_{2} \cdot e_{1} =e_{2}, \\
 \mathcal{C}_{07} \colon\ e_{1}\cdot e_{1} =e_{1}, \quad e_{2}\cdot e_{2}=e_{2}, \\
 \mathcal{C}_{08} \colon\ e_{1}\cdot e_{1} =e_{1}, \quad e_{1}\cdot e_{2} = 2 e_{2}, \quad e_{2}\cdot e_{1} =\tfrac{1}{2} e_{1} + e_2, \quad e_{2}\cdot e_{2}=e_{2}.
\end{gather*}
\end{Lemma}

\begin{Lemma}
\label{pre-Lie aut} The description of the group of automorphisms of every $2
$-dimensional pre-Lie algebra is given as follows:
\begin{enumerate}\itemsep=0pt
\item[$(1)$] If $\phi \in \mathrm{Aut}(\mathcal{C}_{01})$, then $\phi
(e_{1})=e_{1}+\nu e_{2}$ and $\phi (e_{2})=e_{2}$ for $\nu \in \mathbb{C}$.

\item[$(2)$] If $\phi \in \mathrm{Aut}(\mathcal{C}_{02})$, then $\phi
(e_{1})=e_{1}+\nu e_{2}$ and $\phi (e_{2})=e_{2}$ for $\nu \in \mathbb{C}$.

\item[$(3)$] If $\phi \in \mathrm{Aut}(\mathcal{C}_{03})$, then $\phi (e_{1})=\xi
e_{1}+\nu e_{2}$ and $\phi (e_{2})=\xi ^{2}e_{2}$ for $\xi \in \mathbb{C}%
^{\ast }$ and $\nu \in \mathbb{C}$.

\item[$(4)$] If $\phi \in \mathrm{Aut}(\mathcal{C}_{04})$, then $\phi (e_{1})=\xi
e_{1}$ and $\phi (e_{2})=e_{2}$ for $\xi \in \mathbb{C}^{\ast }$.

\item[$(5)$] If $\phi \in \mathrm{Aut}\bigl(\mathcal{C}_{05}^{\alpha \neq 1}\bigr)$, then $%
\phi (e_{1})=e_{1}$ and $\phi (e_{2})=\xi e_{2}$ for $\xi \in \mathbb{C}%
^{\ast }$.

\item[$(6)$] If $\phi \in \mathrm{Aut}\bigl(\mathcal{C}_{05}^{1}\bigr)$, then $\phi
(e_{1})=e_{1}+\nu e_{2}$ and $\phi (e_{2})=\xi e_{2}$ for $\xi \in \mathbb{C%
}^{\ast }$ and $\nu \in \mathbb{C}$.

\item[$(7)$] If $\phi \in \mathrm{Aut}\bigl(\mathcal{C}_{06}^{\alpha \neq 0}\bigr)$, then $%
\phi (e_{1})=e_{1}$ and $\phi (e_{2})=\xi e_{2}$ for $\xi \in \mathbb{C}%
^{\ast }$.

\item[$(8)$] If $\phi \in \mathrm{Aut}\bigl(\mathcal{C}_{06}^{0}\bigr)$, then $\phi
(e_{1})=e_{1}+\nu e_{2}$ and $\phi (e_{2})=\xi e_{2}$ for $\xi \in \mathbb{C%
}^{\ast }$ and $\nu \in \mathbb{C}$.

\item[$(9)$] If $\phi \in \mathrm{Aut}(\mathcal{C}_{07})$, then $\phi \in \mathbb{S}%
_{2}$, i.e., $\phi (e_{1})=e_{1}$, $ \phi (e_{2})=e_{2}$ or $\phi
(e_{1})=e_{2}$, $ \phi (e_{2})=e_{1}.$

\item[$(10)$] If $\phi \in \mathrm{Aut}(\mathcal{C}_{08})$, then $\phi
(e_{1})=e_{1}$, $\phi (e_{2})=e_{2}$ or $\phi (e_{1})=-e_{1}+4e_{2}$, $\phi
(e_{2})=e_{2}$.
\end{enumerate}
\end{Lemma}

\subsection{The algebraic classification of compatible pre-Lie algebras}
The main aim of the present section is to prove the following results.

\begin{Theorem}
\label{compre2} Let $\mathcal{C}$ be a nonzero $2$-dimensional compatible pre-Lie algebra.
Then $\mathcal{C}$ is isomorphic to one {and only one} of the following
algebras:
\begin{gather*}
 \mathcal{C}_{01} \colon\ e_{1} \ast e_{1} = e_{1} + e_{2},\quad e_{2} \ast e_{1} =e_{2},\\
 \mathcal{C}_{02} \colon\ e_{1}\ast e_{1} = e_{1} + e_{2}, \quad e_{1}\ast e_{2} =e_{2}, \\
 \mathcal{C}_{03} \colon\ e_{1}\ast e_{1} =e_{2},\\
 \mathcal{C}_{04} \colon\ e_{2} \ast e_{1} =e_{1}, \\
 \mathcal{C}_{05}^{\alpha} \colon\ e_{1}\ast e_{1} =e_{1},\quad e_{1}\ast e_{2} = \alpha e_{2}, \\
 \mathcal{C}_{06}^{\alpha} \colon\ e_{1} \ast e_{1} =e_{1}, \quad e_{1} \ast e_{2} = \alpha e_{2}, \quad e_{2} \ast e_{1} =e_{2}, \\
 \mathcal{C}_{07} \colon\ e_{1}\ast e_{1} =e_{1},\quad e_{2}\ast e_{2}=e_{2}, \\
 \mathcal{C}_{08} \colon\ e_{1}\ast e_{1} =e_{1}, \quad e_{1}\ast e_{2} = 2 e_{2},\quad e_{2}\ast e_{1} =\tfrac{1}{2} e_{1} + e_2, \quad e_{2}\ast e_{2}=e_{2}, \\
 \mathcal{C}_{13}^{1 }\colon\
e_{1}\cdot e_{1}=e_{2}, \quad
e_{1}\ast e_{1}= e_{1},\quad e_{1}\ast e_{2}=e_{2},\quad e_{2}\ast
e_{1}= e_{2},
\\
 \mathcal{C}_{15}^{\alpha }\colon\
 e_{1}\cdot e_{1}=e_{2}, \quad
 e_{1}\ast e_{1}=\alpha e_{2},
\\
 \mathcal{C}_{16}^{0 }\colon\
 e_{1}\cdot e_{1}=e_{2}, \quad
 e_{1}\ast e_{1}=e_{1},
\\
 \mathcal{C}_{18}^{\alpha }\colon\
 e_{1}\cdot e_{1}=e_{2}, \quad
 e_{1}\ast e_{1}=\alpha e_{2},\quad e_{1}\ast e_{2}=e_{1},\quad e_{2}\ast
e_{1}=e_{1}, \quad e_{2}\ast e_{2}=e_{2},
\\
 \mathcal{C}_{24}^{0 ,\beta ,0 }\colon\
 e_{1}\cdot e_{1}=e_{1}, \quad
 e_{1}\ast e_{1}=\beta e_{1}+e_{2},
\\
 \mathcal{C}_{25}^{0 ,\beta ,0 }\colon\
 e_{1}\cdot e_{1}=e_{1}, \quad
 e_{1}\ast e_{1}=\beta e_{1},
\\
 \mathcal{C}_{29}^{\alpha }\colon\
 e_{1}\cdot e_{1}=e_{1}, \quad
 e_{1}\ast e_{1}=\alpha e_{1}, \quad e_{2}\ast e_{2}=e_{2},
\\
 \mathcal{C}_{30}^{\alpha ,\beta }\colon\
 e_{1}\cdot e_{1}=e_{1}, \quad
 e_{1}\ast e_{1}=\alpha e_{1}+\beta e_{2}, \quad e_{1}\ast e_{2}=e_{1}, \quad
e_{2}\ast e_{1}=e_{1}, \quad
 e_{2}\ast e_{2}=e_{2},
\\
 \mathcal{C}_{31}^{1 ,\beta ,\beta } \colon\
 e_{1}\cdot e_{1}=e_{1}, \quad e_{1}\cdot e_{2}= e_{2}, \quad e_{2}\cdot
e_{1}=e_{2}, \\
\hphantom{\mathcal{C}_{31}^{1 ,\beta ,\beta } \colon} \ e_{1}\ast e_{1}=\beta e_{1}+e_{2}, \quad e_{1}\ast e_{2}=\beta e_{2}, \quad
e_{2}\ast e_{1}=\beta e_{2},
\\
 \mathcal{C}_{32}^{1 ,\beta ,\beta }\colon\
 e_{1}\cdot e_{1}=e_{1}, \quad e_{1}\cdot e_{2}= e_{2}, \quad e_{2}\cdot
e_{1}=e_{2}, \\
\hphantom{\mathcal{C}_{32}^{1 ,\beta ,\beta }\colon} \ e_{1}\ast e_{1}=\beta e_{1}, \quad e_{1}\ast e_{2}=\beta e_{2}, \quad
e_{2}\ast e_{1}=\beta e_{2},
 \\
 \mathcal{C}_{34}^{\alpha ,\beta }\colon\
 e_{1}\cdot e_{1}=e_{1}, \quad e_{1}\cdot e_{2}=e_{2}, \quad e_{2}\cdot
e_{1}=e_{2}, \\
\hphantom{\mathcal{C}_{34}^{\alpha ,\beta }} \ e_{1}\ast e_{1}=\alpha e_{1}, \quad e_{1}\ast e_{2}=\alpha e_{2}, \quad
e_{2}\ast e_{1}=\alpha e_{2}, \quad e_{2}\ast e_{2}=e_{1}+\beta e_{2},
\\
 \mathcal{C}_{35}^{\alpha }\colon\
 e_{1}\cdot e_{1}=e_{1}, \quad e_{1}\cdot e_{2}=e_{2}, \quad e_{2}\cdot
e_{1}=e_{2}, \\
\hphantom{\mathcal{C}_{35}^{\alpha }\colon} \ e_{1}\ast e_{1}=\alpha e_{1}, \quad e_{1}\ast e_{2}=\alpha e_{2}, \quad
e_{2}\ast e_{1}=\alpha e_{2}, \quad e_{2}\ast e_{2}=e_{2},
 \\
 \mathcal{C}_{38}^{\alpha ,\beta ,\gamma }\colon\
 e_{1}\cdot e_{1}=e_{1}, \quad e_{2}\cdot e_{2}=e_{2}, \quad
 e_{1}\ast e_{1}=( \gamma +\beta -\alpha ) e_{1}-\beta e_{2},\\
\hphantom{\mathcal{C}_{38}^{\alpha ,\beta ,\gamma }\colon} \ e_{1}\ast e_{2}=\alpha e_{1}+\beta e_{2}, \quad
 e_{2}\ast e_{1}=\alpha
e_{1}+\beta e_{2}, \quad e_{2}\ast e_{2}=-\alpha e_{1}+\gamma e_{2},
\\
 \mathcal{C}_{39}^{\alpha ,\beta \neq \alpha }\colon\
 e_{1}\cdot e_{1}=e_{1}, \quad e_{2}\cdot e_{2}=e_{2}, \quad
 e_{1}\ast e_{1}=\alpha e_{1}, \quad e_{2}\ast e_{2}=\beta e_{2}.
 \end{gather*}

All algebras are non-isomorphic, except
\begin{gather*}
\mathcal{C}_{24}^{1 ,\beta ,\gamma } \cong \mathcal{C}_{25}^{1 , \beta ,\gamma },\quad
\mathcal{C}_{31}^{0 ,\beta ,\gamma } \cong \mathcal{C}_{32}^{0 , \beta ,\gamma },\quad
\mathcal{C}_{38}^{\alpha ,\beta ,\gamma } \cong \mathcal{C}_{38}^{\beta ,\alpha , -\alpha +\beta +\gamma }, \\
 \mathcal{C}_{39}^{\alpha ,\beta } \cong \mathcal{C}_{39}^{\beta,\alpha},\quad
\mathcal{C}_{40}^{\alpha ,\beta ,\gamma } \cong \mathcal{C}_{40}^{4\gamma -\alpha ,\beta
-8\alpha +16\gamma ,\gamma },\quad
 \mathcal{C}_{41}^{\alpha ,\beta } \cong \mathcal{C}%
_{41}^{ \alpha +4\beta ,-\beta }.
\end{gather*}
\end{Theorem}

\subsubsection[Compatible pre-Lie algebras defined on C\_\{01\}]{Compatible pre-Lie algebras defined on $\boldsymbol{\mathcal{C}_{01}}$}

From the computation of $\mathrm{Z}_{\mathrm{CPL}}^{2}(\mathcal{C}_{01},%
\mathcal{C}_{01})$, the compatible pre-Lie algebra structures defined on~$%
\mathcal{C}_{01}$ are of the form
$e_{1}\cdot e_{1}=e_{1}+e_{2}$, $ e_{2}\cdot e_{1}=e_{2}$, $e_{1}\ast e_{1}=\alpha _{1}e_{1}+\alpha _{2}e_{2}$, $e_{1}\ast
e_{2}=\alpha _{3}e_{2}$, $e_{2}\ast e_{1}=\alpha _{1}e_{2}$.
Then we have the following cases:
\begin{itemize}\itemsep=0pt
\item If $\alpha _{3}\neq 0$, then choose $\nu =-\frac{\alpha _{2}}{\alpha
_{3}}$ and obtain the parametric family
\begin{gather*}
\mathcal{C}_{09}^{\alpha ,\beta \neq 0} \colon \
e_{1}\cdot e_{1}=e_{1}+e_{2},\quad e_{2}\cdot e_{1}=e_{2},\\
\hphantom{\mathcal{C}_{09}^{\alpha ,\beta \neq 0} \colon} \ e_{1}\ast e_{1}=\alpha e_{1},\quad e_{1}\ast e_{2}=\beta e_{2},\quad
e_{2}\ast e_{1}=\alpha e_{2}.
\end{gather*}
The algebras $\mathcal{C}_{09}^{\alpha ,\beta }$ and \smash{$\mathcal{C}_{09}^{\alpha
^{\prime },\beta ^{\prime }}$} are isomorphic if and only if $(\alpha,\beta)= (\alpha^{\prime },\beta ^{\prime })$.

\item If $\alpha _{3}=0$, we obtain the parametric family
\begin{gather*}
\mathcal{C}_{10}^{\alpha ,\beta }\colon \
e_{1}\cdot e_{1}=e_{1}+e_{2},\quad e_{2}\cdot e_{1}=e_{2}, \quad e_{1}\ast e_{1}=\alpha e_{1}+\beta e_{2},\quad e_{2}\ast e_{1}=\alpha e_{2}.
\end{gather*}
The algebras $\mathcal{C}_{10}^{\alpha ,\beta }$ and $\mathcal{C}_{10}^{\alpha
^{\prime },\beta ^{\prime }}$ are isomorphic if and only if $(\alpha,\beta)= (\alpha^{\prime },\beta ^{\prime })$.
\end{itemize}

\subsubsection[Compatible pre-Lie algebras defined on C\_\{02\}]{Compatible pre-Lie algebras defined on $\boldsymbol{\mathcal{C}_{02}}$}

From the computation of $\mathrm{Z}_{\mathrm{CPL}}^{2}(\mathcal{C}_{02},%
\mathcal{C}_{02})$, the compatible pre-Lie algebra structures defined on~$%
\mathcal{C}_{02}$ are of the form
$
e_{1}\cdot e_{1}=e_{1}+e_{2}$, $ e_{1}\cdot e_{2}=e_{2}$, $
e_{1}\ast e_{1}=\alpha _{1}e_{1}+\alpha _{2}e_{2}$, $ e_{1}\ast
e_{2}=\alpha _{3}e_{2}$.
Then we have the following cases:
\begin{itemize}\itemsep=0pt
\item If $\alpha _{1}\neq\alpha _{3}$, then choose $\nu =\frac{\alpha _{2}%
}{\alpha _{1}-\alpha _{3}}$ and obtain the parametric family
\[
\mathcal{C}_{11}^{\alpha ,\beta\neq \alpha }\colon \
e_{1}\cdot e_{1}=e_{1}+e_{2},\quad e_{1}\cdot e_{2}=e_{2},\quad
e_{1}\ast e_{1}=\alpha e_{1},\quad e_{1}\ast e_{2}=\beta e_{2}.
\]
The algebras $\mathcal{C}_{11}^{\alpha ,\beta }$and $\mathcal{C}_{11}^{\alpha
^{\prime },\beta ^{\prime }}$ are isomorphic if and only if $(\alpha,\beta)= (\alpha^{\prime },\beta ^{\prime })$.

\item If $\alpha _{1}=\alpha _{3}$, we obtain the parametric family
\[
\mathcal{C}_{12}^{\alpha ,\beta } \colon \
e_{1}\cdot e_{1}=e_{1}+e_{2},\quad e_{1}\cdot e_{2}=e_{2}, \quad
e_{1}\ast e_{1}=\alpha e_{1}+\beta e_{2},\quad e_{1}\ast e_{2}=\alpha e_{2}.
\]
The algebras $\mathcal{C}_{12}^{\alpha ,\beta }$ and \smash{$\mathcal{C}_{12}^{\alpha
^{\prime },\beta ^{\prime }}$} are isomorphic if and only if $(\alpha,\beta)= (\alpha^{\prime },\beta ^{\prime })$.
\end{itemize}

\subsubsection[Compatible pre-Lie algebras defined on C\_\{03\}]{Compatible pre-Lie algebras defined on $\boldsymbol{\mathcal{C}_{03}}$}

From the computation of $\mathrm{Z}_{\mathrm{CPL}}^{2}(\mathcal{C}_{03},%
\mathcal{C}_{03})$, the compatible pre-Lie algebra structures defined on~$%
\mathcal{C}_{03}$ are of the following forms:
\begin{enumerate}\itemsep=0pt
\item[(1)] $
e_{1}\cdot e_{1}=e_{2}$, $
e_{1}\ast e_{1}=\alpha _{1}e_{1}+\alpha _{2}e_{2}$, $ e_{1}\ast
e_{2}=\alpha _{3}e_{2}$, $ e_{2}\ast e_{1}=\alpha _{1}e_{2}$,
\item[(2)] $
e_{1}\cdot e_{1}=e_{2}$, $
e_{1}\ast e_{1}=\alpha _{4}e_{1}+\alpha _{5}e_{2}$, $ e_{1}\ast
e_{2}=\alpha _{6}e_{2}$,
\item[(3)] $
e_{1}\cdot e_{1}=e_{2}$, $
e_{1}\ast e_{1}=\alpha _{7}e_{1}+\alpha _{8}e_{2}$, $ e_{1}\ast
e_{2}=\alpha _{9}e_{1}$, $ e_{2}\ast e_{1}=\alpha _{9}e_{1}$, $ e_{2}\ast
e_{2}=\alpha _{9}e_{2}$,
\item[(4)] $
e_{1}\cdot e_{1}=e_{2}$, $
e_{1}\ast e_{1}=\alpha _{10}e_{1}+\alpha _{11}e_{2}$, $ e_{1}\ast
e_{2}=2\alpha _{10}e_{2}$,
$
e_{2}\ast e_{1}=\alpha _{12}e_{1}+\alpha
_{10}e_{2}$, $ e_{2}\ast e_{2}=2\alpha _{12}e_{2}$.
\end{enumerate}

We may assume $\alpha _{4}\alpha _{9}\alpha _{12}\neq 0$.
First, we consider the first form. Then we have the following cases:
\begin{itemize}\itemsep=0pt
\item If $\alpha _{3}\neq 0$, then choose $\xi =\alpha^{-1} _{3}$, $\nu =-%
{\alpha _{2}}{\alpha _{3}^{-2}}$ and obtain the parametric family
\[
\mathcal{C}_{13}^{\alpha } \colon \
e_{1}\cdot e_{1}=e_{2},\quad
e_{1}\ast e_{1}=\alpha e_{1},\quad e_{1}\ast e_{2}=e_{2},\quad e_{2}\ast
e_{1}=\alpha e_{2}.
\]
The algebras $\mathcal{C}_{13}^{\alpha }$ and $\mathcal{C}_{13}^{\alpha ^{\prime
}} $ are isomorphic if and only if $\alpha =\alpha ^{\prime }$.

\item If $\alpha _{3}=0,\alpha _{1}\neq 0$, then choose $\xi ={%
\alpha^{-1} _{1}}$ and obtain the parametric family
\[
\mathcal{C}_{14}^{\alpha } \colon \
e_{1}\cdot e_{1}=e_{2},\quad
e_{1}\ast e_{1}=e_{1}+\alpha e_{2},\quad e_{2}\ast e_{1}=e_{2}.
\]
The algebras $\mathcal{C}_{14}^{\alpha }$ and $\mathcal{C}_{14}^{\alpha ^{\prime
}} $ are isomorphic if and only if $\alpha =\alpha ^{\prime }$.

\item If $\alpha _{1}=\alpha _{3}=0$, then we obtain the parametric family
\[
\mathcal{C}_{15}^{\alpha } \colon \
e_{1}\cdot e_{1}=e_{2}, \quad
e_{1}\ast e_{1}=\alpha e_{2}.
\]
The algebras $\mathcal{C}_{15}^{\alpha }$ and $\mathcal{C}_{15}^{\alpha ^{\prime
}} $ are isomorphic if and only if $\alpha =\alpha ^{\prime }$.
\end{itemize}

Second, we consider the second form. Then we have the following cases:
\begin{itemize}\itemsep=0pt
\item If $\alpha _{4}\neq\alpha _{6}$, then choose $\xi ={\alpha^{-1}
_{4}}$, $\nu =\frac{\alpha _{5}}{\alpha _{4}-\alpha _{6}}$
and obtain the parametric family
\[
\mathcal{C}_{16}^{\alpha }\colon \
e_{1}\cdot e_{1}=e_{2}, \quad
e_{1}\ast e_{1}=e_{1},\quad e_{1}\ast e_{2}=\alpha e_{2}.
\]
The algebras $\mathcal{C}_{16}^{\alpha }$ and $\mathcal{C}_{16}^{\alpha ^{\prime
}} $ are isomorphic if and only if $\alpha =\alpha ^{\prime }$.

\item If $\alpha _{4}=\alpha _{6}$, then choose $\xi ={\alpha^{-1} _{4}}$
and obtain the parametric family
\[
\mathcal{C}_{17}^{\alpha }\colon \
e_{1}\cdot e_{1}=e_{2}, \quad
e_{1}\ast e_{1}=e_{1}+\alpha e_{2},\quad e_{1}\ast e_{2}=e_{2}.
\]
The algebras $\mathcal{C}_{17}^{\alpha }$ and $\mathcal{C}_{17}^{\alpha ^{\prime}} $ are isomorphic if and only if $\alpha =\alpha ^{\prime }$.
\end{itemize}

Third, we consider the third form. Then, we choose \smash{$\xi ={\alpha^{-\frac 12}
_{9}}$}, \smash{$\nu =-\frac{\alpha _{7} \alpha _{9}^{-\frac{3}{2}}}{2 }$} and
obtain the parametric family
\[
\mathcal{C}_{18}^{\alpha }\colon \
e_{1}\cdot e_{1}=e_{2}, \quad
e_{1}\ast e_{1}=\alpha e_{2},\quad e_{1}\ast e_{2}=e_{1},\quad e_{2}\ast
e_{1}=e_{1},\quad e_{2}\ast e_{2}=e_{2}.
\]
The algebras $\mathcal{C}_{18}^{\alpha }$ and $\mathcal{C}_{18}^{\alpha ^{\prime}} $ are isomorphic if and only if $\alpha =\alpha ^{\prime }$.

Finally, we consider the fourth form. Then, we choose \smash{$\xi =
\alpha^{-\frac 12} _{12}$}, \smash{$\nu =-{\alpha _{10}}{\alpha _{12}^{-\frac{3}{2}}}$} and
obtain the parametric family
\[
\mathcal{C}_{19}^{\alpha }\colon \
e_{1}\cdot e_{1}=e_{2}, \quad
e_{1}\ast e_{1}=\alpha e_{2},\quad e_{2}\ast e_{1}=e_{1},\quad e_{2}\ast
e_{2}=2e_{2}.
\]

The algebras $\mathcal{C}_{19}^{\alpha }$ and $\mathcal{C}_{19}^{\alpha ^{\prime
}} $ are isomorphic if and only if $\alpha =\alpha ^{\prime }$.

\subsubsection[Compatible pre-Lie algebras defined on C\_\{04\}]{Compatible pre-Lie algebras defined on $\boldsymbol{\mathcal{C}_{04}}$}

From the computation of $\mathrm{Z}_{\mathrm{CPL}}^{2}(\mathcal{C}_{04},
\mathcal{C}_{04})$, the compatible pre-Lie algebra structures defined on~$\mathcal{C}_{04}$ are of the following forms:
\begin{enumerate}\itemsep=0pt
\item[(1)] $e_{2}\cdot e_{1}=e_{1}$,
$e_{2}\ast e_{1}=\alpha _{1}e_{1}$, $e_{2}\ast e_{2}=\alpha
_{2}e_{1}+\alpha _{3}e_{2}$,
\item[(2)] $e_{2}\cdot e_{1}=e_{1}$,
$e_{1}\ast e_{2}=\alpha _{4}e_{1}$, $e_{2}\ast e_{1}=\alpha _{5}e_{1}$, $e_{2}\ast e_{2}=\alpha _{6}e_{1}+\alpha _{4}e_{2}$.
\end{enumerate}

We may assume $\alpha _{4}\neq 0$.
First, we study the first form. Then
we have the following cases:
\begin{itemize}\itemsep=0pt
\item If $\alpha _{2}\neq 0$, then choose $\xi =\alpha _{2}$ and
obtain the parametric family
\[
\mathcal{C}_{20}^{\alpha ,\beta } \colon \
e_{2}\cdot e_{1}=e_{1}, \quad
e_{2}\ast e_{1}=\alpha e_{1}, \quad e_{2}\ast e_{2}=e_{1}+\beta e_{2}.
\]
The algebras \smash{$\mathcal{C}_{20}^{\alpha ,\beta }$} and \smash{$\mathcal{C}_{20}^{\alpha
^{\prime },\beta ^{\prime }}$} are isomorphic if and only if $(\alpha,\beta)= (\alpha^{\prime },\beta ^{\prime })$.

\item If $\alpha _{2}=0$, we obtain the parametric family
\[
\mathcal{C}_{21}^{\alpha ,\beta } \colon \
e_{2}\cdot e_{1}=e_{1}, \quad
e_{2}\ast e_{1}=\alpha e_{1},\quad e_{2}\ast e_{2}=\beta e_{2}.
\]
The algebras \smash{$\mathcal{C}_{21}^{\alpha ,\beta }$} and \smash{$\mathcal{C}_{21}^{\alpha
^{\prime },\beta ^{\prime }}$} are isomorphic if and only if $(\alpha,\beta)= (\alpha^{\prime },\beta ^{\prime })$.
\end{itemize}

Second, we consider the second form. Then we have the following cases:
\begin{itemize}\itemsep=0pt
\item If $\alpha _{6}\neq 0$, then choose $\xi =\alpha _{6}$ and
obtain the parametric family
\[
\mathcal{C}_{22}^{\alpha \neq 0,\beta } \colon \
e_{2}\cdot e_{1}=e_{1}, \quad
e_{1}\ast e_{2}=\alpha e_{1}, \quad e_{2}\ast e_{1}=\beta e_{1},\quad
e_{2}\ast e_{2}=e_{1}+\alpha e_{2}.
\]
The algebras \smash{$\mathcal{C}_{22}^{\alpha ,\beta }$} and \smash{$\mathcal{C}_{22}^{\alpha
^{\prime },\beta ^{\prime }}$} are isomorphic if and only if $(\alpha,\beta)= (\alpha^{\prime },\beta ^{\prime })$.

\item If $\alpha _{6}=0$, we obtain the parametric family
\[
\mathcal{C}_{23}^{\alpha \neq 0,\beta } \colon \
e_{2}\cdot e_{1}=e_{1}, \quad
e_{1}\ast e_{2}=\alpha e_{1}, \quad e_{2}\ast e_{1}=\beta e_{1}, \quad
e_{2}\ast e_{2}=\alpha e_{2}.
\]
The algebras \smash{$\mathcal{C}_{23}^{\alpha ,\beta }$} and \smash{$\mathcal{C}_{23}^{\alpha
^{\prime },\beta ^{\prime }}$} are isomorphic if and only if $(\alpha,\beta)= (\alpha^{\prime },\beta ^{\prime })$.
\end{itemize}

\subsubsection[Compatible pre-Lie algebras defined on C\_\{05\}\^{}\{alpha not = 0,1/2,1\}]{Compatible pre-Lie algebras defined on $\boldsymbol{\mathcal{C}_{05}^{\alpha \not= 0,\frac{1}{2},1}}$}

From the computation of $\mathrm{Z}_{\mathrm{CPL}}^{2}(\mathcal{C}%
_{05}^{\alpha },\mathcal{C}_{05}^{\alpha })$, the compatible pre-Lie algebra
structures defined on~$\mathcal{C}_{05}^{\alpha }$ are of the form
$e_{1}\cdot e_{1}=e_{1}$, $e_{1}\cdot e_{2}=\alpha e_{2}$, $e_{1}\ast e_{1}=\alpha _{1}e_{1}+\alpha _{2}e_{2}$, $e_{1}\ast
e_{2}=\alpha _{3}e_{2}$.
Then we have the following cases:
\begin{itemize}\itemsep=0pt
\item If $\alpha _{2}\neq 0$, then choose $\xi =\alpha _{2}$ and
obtain the parametric family
\[
\mathcal{C}_{24}^{\alpha ,\beta ,\gamma }\colon \
e_{1}\cdot e_{1}=e_{1}, \quad e_{1}\cdot e_{2}=\alpha e_{2}, \quad
e_{1}\ast e_{1}=\beta e_{1}+e_{2}, \quad e_{1}\ast e_{2}=\gamma e_{2}.
\]
The algebras \smash{$\mathcal{C}_{24}^{\alpha ,\beta ,\gamma }$} and \smash{$\mathcal{C}%
_{24}^{\alpha ^{\prime }\beta ^{\prime },\gamma ^{\prime }}$} are isomorphic
if and only if $(\alpha, \beta, \gamma)=(\alpha ^{\prime },\beta ^{\prime },
\gamma ^{\prime })$.

\item If $\alpha _{2}=0$, we obtain the parametric family
\[
\mathcal{C}_{25}^{\alpha ,\beta ,\gamma } \colon \
e_{1}\cdot e_{1}=e_{1}, \quad e_{1}\cdot e_{2}=\alpha e_{2}, \quad
e_{1}\ast e_{1}=\beta e_{1}, \quad e_{1}\ast e_{2}=\gamma e_{2}.
\]
The algebras \smash{$\mathcal{C}_{25}^{\alpha ,\beta ,\gamma }$} and \smash{$\mathcal{C}%
_{25}^{\alpha ^{\prime },\beta ^{\prime },\gamma ^{\prime }}$} are isomorphic
if and only if $(\alpha, \beta, \gamma)=(\alpha ^{\prime },\beta ^{\prime },
\gamma ^{\prime })$.
\end{itemize}

\subsubsection[Compatible pre-Lie algebras defined on C\_\{05\}\^{}\{1/2\}]{Compatible pre-Lie algebras defined on $\boldsymbol{\mathcal{C}_{05}^{\frac{1}{2}}}$}

From the computation of \smash{$\mathrm{Z}_{\mathrm{CPL}}^{2}\bigl(\mathcal{C}_{05}^{%
\frac{1}{2}},\mathcal{C}_{05}^{\frac{1}{2}}\bigr)$}, the compatible pre-Lie
algebra structures defined on~\smash{$\mathcal{C}_{05}^{\frac{1}{2}}$} are of the following forms:
\begin{enumerate}\itemsep=0pt
\item[(1)] $e_{1}\cdot e_{1}=e_{1}$, $e_{1}\cdot e_{2}=\frac{1}{2}e_{2}$,
$e_{1}\ast e_{1}=\alpha _{1}e_{1}+\alpha _{2}e_{2}$, $e_{1}\ast
e_{2}=\alpha _{3}e_{2}$,

\item[(2)] $e_{1}\cdot e_{1}=e_{1}$, $e_{1}\cdot e_{2}=\frac{1}{2}e_{2}$,
$e_{1}\ast e_{1}=2\alpha _{4}e_{1}$, $e_{1}\ast e_{2}=\alpha _{4}e_{2}$,
$e_{2}\ast e_{2}=\alpha _{5}e_{1}$,

\item[(3)] $e_{1}\cdot e_{1}=e_{1}$, $e_{1}\cdot e_{2}=\frac{1}{2}e_{2}$,
$e_{1}\ast e_{1}=2\alpha _{6}e_{1}$, $e_{1}\ast e_{2}=\alpha
_{7}e_{1}+\alpha _{6}e_{2}$, $e_{2}\ast e_{1}=2\alpha _{7}e_{1}$, $%
e_{2}\ast e_{2}=\alpha _{8}e_{1}+\alpha _{7}e_{2}$.
\end{enumerate}

We may assume $\alpha _{5}\alpha _{7}\neq 0$.
If the compatible pre-Lie
algebra structures defined on \smash{$\mathcal{C}_{05}^{\frac{1}{2}}$} is of the
first form, then we obtain the algebras \smash{$\mathcal{C}_{24}^{\frac{1}{2%
},\beta ,\gamma }$} and \smash{$\mathcal{C}_{25}^{\frac{1}{2},\beta ,\gamma
} $}.

Assume now that the compatible pre-Lie algebra structures defined on \smash{$%
\mathcal{C}_{05}^{\frac{1}{2}}$} is of the second form. Then, choose \smash{$\xi =%
{\alpha^{-\frac 12} _{5}}$} and obtain the parametric family
\[
\mathcal{C}_{26}^{\alpha } \colon \
e_{1}\cdot e_{1}=e_{1}, \quad e_{1}\cdot e_{2}=\frac{1}{2}e_{2}, \quad
e_{1}\ast e_{1}=2\alpha e_{1},\quad e_{1}\ast e_{2}=\alpha e_{2},\quad
e_{2}\ast e_{2}=e_{1}.
\]
The algebras $\mathcal{C}_{26}^{\alpha }$ and $\mathcal{C}_{26}^{\alpha ^{\prime
}}$ are isomorphic if and only if $\alpha =\alpha ^{\prime }$.

Finally, assume that the compatible pre-Lie algebra structures defined on \smash{$%
\mathcal{C}_{05}^{\frac{1}{2}}$} is of the third form. Then, choose $\xi =%
{\alpha^{-1} _{7}}$ and obtain the parametric family
\begin{gather*}
\mathcal{C}_{27}^{\alpha ,\beta } \colon \
e_{1}\cdot e_{1}=e_{1},\quad e_{1}\cdot e_{2}=\tfrac{1}{2}e_{2}, \quad
e_{1}\ast e_{1}=2\alpha e_{1}, \\
\hphantom{\mathcal{C}_{27}^{\alpha ,\beta } \colon} \ e_{1}\ast e_{2}=e_{1}+\alpha e_{2}, \quad
e_{2}\ast e_{1}=2e_{1},\quad e_{2}\ast e_{2}=\beta e_{1}+e_{2}.
\end{gather*}
The algebras $\mathcal{C}_{27}^{\alpha ,\beta }$ and $\mathcal{C}_{27}^{\alpha
^{\prime },\beta ^{\prime }}$ are isomorphic if and only if $(\alpha,\beta)= (\alpha^{\prime },\beta ^{\prime })$.

\subsubsection[Compatible pre-Lie algebras defined on C\_\{05\}\^{}1]{Compatible pre-Lie algebras defined on $\boldsymbol{\mathcal{C}_{05}^{1}}$}

From the computation of $\mathrm{Z}_{\mathrm{CPL}}^{2}\bigl(\mathcal{C}_{05}^{1},%
\mathcal{C}_{05}^{1}\bigr)$, the compatible pre-Lie algebra structures defined on~$\mathcal{C}_{05}^{1}$ are of the following forms:
\begin{enumerate}\itemsep=0pt
\item[(1)] $e_{1}\cdot e_{1}=e_{1}$, $e_{1}\cdot e_{2}=e_{2}$,
$e_{1}\ast e_{1}=\alpha _{1}e_{1}+\alpha _{2}e_{2}$, $e_{1}\ast
e_{2}=\alpha _{3}e_{2}$,
\item[(2)] $e_{1}\cdot e_{1}=e_{1}$, $e_{1}\cdot e_{2}=e_{2}$,
$e_{1}\ast e_{1}=\alpha _{4}e_{1}$, $e_{1}\ast e_{2}=\alpha _{4}e_{2}$, $e_{2}\ast e_{1}=\alpha _{5}e_{1}$, $e_{2}\ast e_{2}=\alpha _{5}e_{2}$.
\end{enumerate}

We may assume $\alpha _{5}\neq 0$.
First, suppose that the compatible
pre-Lie algebra structures defined on $\mathcal{C}_{05}^{1}$ is of the first
form. Then we have the following cases:
\begin{itemize}\itemsep=0pt
\item If $\alpha _{1}\neq\alpha _{3}$, then choose $\nu =\frac{\alpha _{2}%
}{\alpha _{1}-\alpha _{3}}$ and obtain the parametric family \smash{$\mathcal{C}%
_{25}^{1,\beta ,\gamma \neq \beta }$}.

\item If $\alpha _{1}=\alpha _{3}$, then we will consider the following two cases:
\begin{itemize}\itemsep=0pt
\item If $\alpha _{2}\neq 0$, then choose $\xi =\alpha _{2}$ and
obtain the parametric family \smash{$\mathcal{C}%
_{24}^{1,\beta ,\beta }$}.
\item If $\alpha _{2}=0$, then we obtain the parametric family \smash{$\mathcal{C}%
_{25}^{1,\beta ,\beta }$}.
\end{itemize}
\end{itemize}

Assume now that the compatible pre-Lie algebra structures defined on $%
\mathcal{C}_{05}^{1}$ is of the second form. Then, choose $\xi ={%
\alpha^{-1} _{5}}$, $\nu =-{\alpha _{4}}{\alpha^{-1} _{5}}$ and obtain the algebra
\[
\mathcal{C}_{28} \colon \
e_{1}\cdot e_{1}=e_{1}, \quad e_{1}\cdot e_{2}=e_{2}, \quad
e_{2}\ast e_{1}=e_{1},\quad e_{2}\ast e_{2}=e_{2}.
\]

\subsubsection[Compatible pre-Lie algebras defined on C\_\{05\}\^{}0]{Compatible pre-Lie algebras defined on $\boldsymbol{\mathcal{C}_{05}^{0}}$}

From the computation of $\mathrm{Z}_{\mathrm{CPL}}^{2}\bigl(\mathcal{C}_{05}^{0},%
\mathcal{C}_{05}^{0}\bigr)$, the compatible pre-Lie algebra structures defined on~$\mathcal{C}_{05}^{0}$ are of the following forms:
\begin{enumerate}\itemsep=0pt
\item[(1)] $e_{1}\cdot e_{1}=e_{1}$,
$e_{1}\ast e_{1}=\alpha _{1}e_{1}+\alpha _{2}e_{2}$, $e_{1}\ast
e_{2}=\alpha _{3}e_{2}$,
\item[(2)] $e_{1}\cdot e_{1}=e_{1}$, $e_{1}\ast e_{1}=\alpha _{4}e_{1}$, $e_{2}\ast e_{2}=\alpha _{5}e_{2}$,
\item[(3)] $e_{1}\cdot e_{1}=e_{1}$, $e_{1}\ast e_{1}=\alpha _{6}e_{1}+\alpha _{7}e_{2}$, $e_{1}\ast
e_{2}=\alpha _{8}e_{1}$, $e_{2}\ast e_{1}=\alpha _{8}e_{1}$, $e_{2}\ast
e_{2}=\alpha _{8}e_{2}$.
\end{enumerate}

We may assume $\alpha _{5}\alpha _{8}\neq 0$.
If the compatible pre-Lie
algebra structures defined on $\mathcal{C}_{05}^{0}$ is of the first form,
then we obtain the algebras \smash{$\mathcal{C}_{24}^{0,\beta ,\gamma }$}
and \smash{$\mathcal{C}_{25}^{0,\beta ,\gamma }$}.

Assume now that the
compatible pre-Lie algebra structures defined on $\mathcal{C}_{05}^{0}$ are
of the second form. Then, choose $\xi ={\alpha^{-1} _{5}}$ and obtain the
parametric family
\[
\mathcal{C}_{29}^{\alpha }\colon \
e_{1}\cdot e_{1}=e_{1}, \quad
e_{1}\ast e_{1}=\alpha e_{1}, \quad e_{2}\ast e_{2}=e_{2}.
\]
The algebras $\mathcal{C}_{29}^{\alpha }$ are $\mathcal{C}_{29}^{\alpha ^{\prime
}} $ are isomorphic if and only if $\alpha =\alpha ^{\prime }$.

Finally,
assume that the compatible pre-Lie algebra structures defined on $\mathcal{C}%
_{05}^{0}$ is of the third form. Then, choose $\xi ={\alpha^{-1} _{8}}$
and obtain the parametric family
\[
\mathcal{C}_{30}^{\alpha ,\beta } \colon \
e_{1}\cdot e_{1}=e_{1}, \quad
e_{1}\ast e_{1}=\alpha e_{1}+\beta e_{2}, \quad e_{1}\ast e_{2}=e_{1}, \quad
e_{2}\ast e_{1}=e_{1}, \quad e_{2}\ast e_{2}=e_{2}.
\]
The algebras \smash{$\mathcal{C}_{30}^{\alpha ,\beta }$} and \smash{$\mathcal{C}_{30}^{\alpha
^{\prime },\beta ^{\prime }}$} are isomorphic if and only if $(\alpha,\beta)= (\alpha^{\prime },\beta ^{\prime })$.

\subsubsection[Compatible pre-Lie algebras defined on C\_\{06\}\^{}\{alpha not= 0,1,2\}]{Compatible pre-Lie algebras defined on $\boldsymbol{\mathcal{C}_{06}^{\alpha \neq 0,1,2}}$}

From the computation of $\mathrm{Z}_{\mathrm{CPL}}^{2}(\mathcal{C}%
_{06}^{\alpha },\mathcal{C}_{06}^{\alpha })$, the compatible pre-Lie algebra
structures defined on~$\mathcal{C}_{06}^{\alpha }$ are of the form
$e_{1}\cdot e_{1}=e_{1}$, $e_{1}\cdot e_{2}=\alpha e_{2}$, $e_{2}\cdot
e_{1}=e_{2}$, $e_{1}\ast e_{1}=\alpha _{1}e_{1}+\alpha _{2}e_{2}$, $e_{1}\ast
e_{2}=\alpha _{3}e_{2}$, $e_{2}\ast e_{1}=\alpha _{1}e_{2}$.
Then we have the following cases:
\begin{itemize}\itemsep=0pt
\item If $\alpha _{2}\neq 0$, then choose $\xi =\alpha _{2}$ and obtain the
parametric family
\begin{gather*}
\mathcal{C}_{31}^{\alpha ,\beta ,\gamma } \colon \
e_{1}\cdot e_{1}=e_{1}, \quad e_{1}\cdot e_{2}=\alpha e_{2}, \quad e_{2}\cdot
e_{1}=e_{2}, \\
\hphantom{\mathcal{C}_{31}^{\alpha ,\beta ,\gamma } \colon} \
e_{1}\ast e_{1}=\beta e_{1}+e_{2},\quad e_{1}\ast e_{2}=\gamma e_{2},\quad
e_{2}\ast e_{1}=\beta e_{2}.
\end{gather*}
The algebras \smash{$\mathcal{C}_{31}^{\alpha ,\beta ,\gamma }$} and \smash{$\mathcal{C}%
_{31}^{\alpha ^{\prime },\beta ^{\prime },\gamma ^{\prime }}$} are isomorphic
if and only if $(\alpha, \beta, \gamma)=(\alpha ^{\prime },\beta ^{\prime },
\gamma ^{\prime })$.

\item If $\alpha _{2}=0$, we obtain the parametric family
\begin{gather*}
\mathcal{C}_{32}^{\alpha ,\beta ,\gamma } \colon \
e_{1}\cdot e_{1}=e_{1},\quad e_{1}\cdot e_{2}=\alpha e_{2},\quad e_{2}\cdot
e_{1}=e_{2}, \\
\hphantom{\mathcal{C}_{32}^{\alpha ,\beta ,\gamma } \colon} \
e_{1}\ast e_{1}=\beta e_{1},\quad e_{1}\ast e_{2}=\gamma e_{2},\quad
e_{2}\ast e_{1}=\beta e_{2}.
\end{gather*}
The algebras \smash{$\mathcal{C}_{32}^{\alpha ,\beta ,\gamma }$}, \smash{$\mathcal{C}%
_{32}^{\alpha ^{\prime },\beta ^{\prime },\gamma ^{\prime }}$} are isomorphic
if and only if $(\alpha, \beta, \gamma)=(\alpha ^{\prime },\beta ^{\prime },
\gamma ^{\prime })$.
\end{itemize}

\subsubsection[Compatible pre-Lie algebras defined on C\_\{06\}\^{}0]{Compatible pre-Lie algebras defined on $\boldsymbol{\mathcal{C}_{06}^{0}}$}

From the computation of $\mathrm{Z}_{\mathrm{CPL}}^{2}\bigl(\mathcal{C}_{06}^{0},%
\mathcal{C}_{06}^{0}\bigr)$, the compatible pre-Lie algebra structures defined on~$\mathcal{C}_{06}^{0}$ are of the following forms:
\begin{enumerate}\itemsep=0pt
\item[(1)] $e_{1}\cdot e_{1}=e_{1}$, $e_{2}\cdot e_{1}=e_{2}$,
$e_{1}\ast e_{1}=\alpha _{1}e_{1}+\alpha _{2}e_{2}$, $e_{1}\ast
e_{2}=\alpha _{3}e_{2}$, $e_{2}\ast e_{1}=\alpha _{1}e_{2}$,

\item[(2)] $e_{1}\cdot e_{1}=e_{1}$, $e_{2}\cdot e_{1}=e_{2}$,
$e_{1}\ast e_{1}=\alpha _{4}e_{1}$, $e_{1}\ast e_{2}=\alpha _{5}e_{1}$, $
e_{2}\ast e_{1}=\alpha _{4}e_{2}$, $e_{2}\ast e_{2}=\alpha _{5}e_{2}$.
\end{enumerate}

We may assume $\alpha _{5}\neq 0$.
Let us first consider the first form.
Then we have the following cases:
\begin{itemize}\itemsep=0pt
\item If $\alpha _{3}\neq 0$, then $\nu =-\frac{\alpha _{2}}{\alpha _{3}}$
choose and obtain the parametric family \smash{$\mathcal{C}_{32}^{0,\beta,\gamma \neq 0}$}.

\item If $\alpha _{3}=0$, then we consider the following two cases:
\begin{itemize}\itemsep=0pt
\item If $\alpha _{2}\neq 0$, then choose $\xi =\alpha _{2}$ and obtain the
parametric family \smash{$ {\mathcal{C}_{31}^{0,\beta, 0}}$}.
\item If $\alpha _{2}=0$, then we obtain the parametric family \smash{$\mathcal{C}%
_{32}^{0,\beta ,0}$}.
\end{itemize}
\end{itemize}

Now, we consider the second form. Then choose $\xi ={\alpha^{-1} _{5}}$, $\nu =-{\alpha _{4}}{\alpha^{-1} _{5}}$ and obtain the algebra
\[
{\mathcal{C}_{33}}\colon \
e_{1}\cdot e_{1}=e_{1}, \quad e_{2}\cdot e_{1}=e_{2}, \quad
e_{1}\ast e_{2}=e_{1}, \quad e_{2}\ast e_{2}=e_{2}.
\]

\subsubsection[Compatible pre-Lie algebras defined on C\_\{06\}\^{}1]{Compatible pre-Lie algebras defined on $\boldsymbol{\mathcal{C}_{06}^{1}}$}

From the computation of $\mathrm{Z}_{\mathrm{CPL}}^{2}\bigl(\mathcal{C}_{06}^{1},%
\mathcal{C}_{06}^{1}\bigr)$, the compatible pre-Lie algebra structures defined on~$\mathcal{C}_{06}^{1}$ are of the following forms:
\begin{enumerate}\itemsep=0pt
\item[(1)] $e_{1}\cdot e_{1}=e_{1}$, $e_{1}\cdot e_{2}=e_{2}$, $e_{2}\cdot
e_{1}=e_{2}$,
$e_{1}\ast e_{1}=\alpha _{1}e_{1}+\alpha _{2}e_{2}$, $e_{1}\ast
e_{2}=\alpha _{3}e_{2}$, $e_{2}\ast e_{1}=\alpha _{1}e_{2}$,

\item[(2)] $e_{1}\cdot e_{1}=e_{1}$, $e_{1}\cdot e_{2}=e_{2}$, $e_{2}\cdot
e_{1}=e_{2}$,
$e_{1}\ast e_{1}=\alpha _{4}e_{1}$, $e_{1}\ast e_{2}=\alpha _{4}e_{2}$, $%
e_{2}\ast e_{1}=\alpha _{4}e_{2}$, $e_{2}\ast e_{2}=\alpha
_{5}e_{1}+\alpha _{6}e_{2}$.
\end{enumerate}

We may assume $( \alpha _{5},\alpha _{6}) \neq ( 0,0) $.
If the compatible pre-Lie algebra structures defined on $\mathcal{C}%
_{06}^{1}$ is of the first form, then we obtain the algebras \smash{$\mathcal{C}%
_{31}^{1,\beta ,\gamma }$} and \smash{$\mathcal{C}_{32}^{1,\beta
,\gamma }$}.

Assume now that the compatible pre-Lie algebra structures
defined on $\mathcal{C}_{06}^{1}$ is of the second form. Then we have the
following cases:
\begin{itemize}\itemsep=0pt
\item If $\alpha _{5}\neq 0$, then choose \smash{$\xi ={\alpha^{-\frac 12} _{5}}$}
and obtain the parametric family
\begin{gather*}
\mathcal{C}_{34}^{\alpha ,\beta } \colon \
e_{1}\cdot e_{1}=e_{1},\quad e_{1}\cdot e_{2}=e_{2},\quad e_{2}\cdot
e_{1}=e_{2}, \\
\hphantom{\mathcal{C}_{34}^{\alpha ,\beta } \colon} \
e_{1}\ast e_{1}=\alpha e_{1},\quad e_{1}\ast e_{2}=\alpha e_{2},\quad
e_{2}\ast e_{1}=\alpha e_{2},\quad e_{2}\ast e_{2}=e_{1}+\beta e_{2}.
\end{gather*}
The algebras \smash{$\mathcal{C}_{34}^{\alpha ,\beta }$} and \smash{$\mathcal{C}_{34}^{\alpha
^{\prime },\beta ^{\prime }}$} are isomorphic if and only if $(\alpha,\beta)= (\alpha^{\prime },\beta ^{\prime })$.

\item If $\alpha _{5}=0$, then choose $\xi = {\alpha^{-1} _{6}}$ and
obtain the parametric family
\begin{gather*}
\mathcal{C}_{35}^{\alpha } \colon \
e_{1}\cdot e_{1}=e_{1},\quad e_{1}\cdot e_{2}=e_{2},\quad e_{2}\cdot
e_{1}=e_{2}, \\
\hphantom{\mathcal{C}_{35}^{\alpha } \colon} \
e_{1}\ast e_{1}=\alpha e_{1},\quad e_{1}\ast e_{2}=\alpha e_{2},\quad
e_{2}\ast e_{1}=\alpha e_{2},\quad e_{2}\ast e_{2}=e_{2}.
\end{gather*}
The algebras $\mathcal{C}_{35}^{\alpha }$ and $\mathcal{C}_{35}^{\alpha ^{\prime }}
$ are isomorphic if and only if $\alpha =\alpha ^{\prime }$.
\end{itemize}

\subsubsection[Compatible pre-Lie algebras defined on C\_\{06\}\^{}2]{Compatible pre-Lie algebras defined on $\boldsymbol{\mathcal{C}_{06}^{2}}$}

From the computation of $\mathrm{Z}_{\mathrm{CPL}}^{2}\bigl(\mathcal{C}_{06}^{2},%
\mathcal{C}_{06}^{2}\bigr)$, the compatible pre-Lie algebra structures defined on~$\mathcal{C}_{06}^{2}$ are of the following forms:
\begin{enumerate}\itemsep=0pt
\item[(1)] $e_{1}\cdot e_{1}=e_{1}$, $e_{1}\cdot e_{2}=2e_{2}$, $e_{2}\cdot
e_{1}=e_{2}$,
$e_{1}\ast e_{1}=\alpha _{1}e_{1}+\alpha _{2}e_{2}$, $e_{1}\ast
e_{2}=\alpha _{3}e_{2}$, $e_{2}\ast e_{1}=\alpha _{1}e_{2}$,

\item[(2)] $e_{1}\cdot e_{1}=e_{1}$, $e_{1}\cdot e_{2}=2e_{2}$, $e_{2}\cdot
e_{1}=e_{2}$, $e_{1}\ast e_{1}=\alpha _{4}e_{1}+\alpha _{5}e_{2}$, $e_{1}\ast
e_{2}=2\alpha _{4}e_{2}$, $e_{2}\ast e_{1}=\alpha _{6}e_{1}+\alpha
_{4}e_{2}$, $e_{2}\ast e_{2}=2\alpha _{6}e_{2}$,

\item[(3)] $e_{1}\cdot e_{1}=e_{1}$, $e_{1}\cdot e_{2}=2e_{2}$, $e_{2}\cdot
e_{1}=e_{2}$,
$e_{1}\ast e_{1}=\alpha _{7}e_{1}$, $e_{1}\ast e_{2}=\alpha
_{8}e_{1}+2\alpha _{7}e_{2}$, $e_{2}\ast e_{1}=2\alpha _{8}e_{1}+\alpha
_{7}e_{2}$, $e_{2}\ast e_{2}=\alpha _{8}e_{2}$.
\end{enumerate}

We may assume $\alpha _{6}\alpha _{8}\neq 0$.
If the compatible pre-Lie
algebra structures defined on $\mathcal{C}_{06}^{2}$ is of the first form,
then we obtain the algebras \smash{$\mathcal{C}_{31}^{2,\beta ,\gamma }$}
and \smash{$\mathcal{C}_{32}^{2,\beta ,\gamma }$}.

Assume now that the
compatible pre-Lie algebra structures defined on $\mathcal{C}_{06}^{2}$ are
of the second form. Then, choose $\xi ={\alpha^{-1} _{6}}$ and obtain the
parametric family
\begin{gather*}
\mathcal{C}_{36}^{\alpha ,\beta } \colon \
e_{1}\cdot e_{1}=e_{1}, \quad e_{1}\cdot e_{2}=2e_{2}, \quad e_{2}\cdot
e_{1}=e_{2}, \quad e_{1}\ast e_{1}=\alpha e_{1}+\beta e_{2}, \\
\hphantom{\mathcal{C}_{36}^{\alpha ,\beta } \colon} \
 e_{1}\ast e_{2}=2\alpha e_{2}, \quad e_{2}\ast e_{1}=e_{1}+\alpha e_{2}, \quad e_{2}\ast e_{2}=2e_{2}.
\end{gather*}
The algebras \smash{$\mathcal{C}_{36}^{\alpha ,\beta }$} and \smash{$\mathcal{C}_{36}^{\alpha
^{\prime },\beta ^{\prime }}$} are isomorphic if and only if $(\alpha,\beta)= (\alpha^{\prime },\beta ^{\prime })$.

Finally, assume that the compatible
pre-Lie algebra structures defined on $\mathcal{C}_{06}^{2}$ are of the third
form. Then, choose $\xi = {\alpha^{-1} _{8}}$ and obtain the parametric
family
\begin{gather*}
\mathcal{C}_{37}^{\alpha } \colon \
e_{1}\cdot e_{1}=e_{1}, \quad e_{1}\cdot e_{2}=2e_{2}, \quad e_{2}\cdot
e_{1}=e_{2}, \\
\hphantom{\mathcal{C}_{37}^{\alpha } \colon} \
e_{1}\ast e_{1}=\alpha e_{1}, \quad e_{1}\ast e_{2}=e_{1}+2\alpha e_{2}, \quad
e_{2}\ast e_{1}=2e_{1}+\alpha e_{2}, \quad e_{2}\ast e_{2}=e_{2}.
\end{gather*}
The algebras $\mathcal{C}_{37}^{\alpha }$ and $\mathcal{C}_{37}^{\alpha ^{\prime }}
$ are isomorphic if and only if $\alpha =\alpha ^{\prime }$.

\subsubsection[Compatible pre-Lie algebras defined on C\_\{07\}]{Compatible pre-Lie algebras defined on $\boldsymbol{\mathcal{C}_{07}}$}

From the computation of $\mathrm{Z}_{\mathrm{CPL}}^{2}(\mathcal{C}_{07},
\mathcal{C}_{07})$, the compatible pre-Lie algebra structures defined on~$\mathcal{C}_{07}$ are of the following forms:
\begin{enumerate}\itemsep=0pt
\item[(1)] $e_{1}\cdot e_{1}=e_{1}$, $e_{2}\cdot e_{2}=e_{2}$,
$e_{1}\ast e_{1}=( \alpha _{3}+\alpha _{2}-\alpha _{1})
e_{1}-\alpha _{2}e_{2}$, $e_{1}\ast e_{2}=\alpha _{1}e_{1}+\alpha
_{2}e_{2}$, $e_{2}\ast e_{1}=\alpha _{1}e_{1}+\alpha _{2}e_{2}$, $e_{2}\ast e_{2}=-\alpha _{1}e_{1}+\alpha _{3}e_{2}$,
\item[(2)] $e_{1}\cdot e_{1}=e_{1}$, $e_{2}\cdot e_{2}=e_{2}$,
$e_{1}\ast e_{1}=\alpha _{4}e_{1}$, $e_{2}\ast e_{2}=\alpha _{5}e_{2}$.
\end{enumerate}

We may assume that $\alpha _{4}\neq \alpha _{5}$.
Suppose first that the
compatible pre-Lie algebra structures defined on $\mathcal{C}_{07}$ has the
first form. Then we obtain the algebras
\begin{gather*}
\mathcal{C}_{38}^{\alpha ,\beta ,\gamma } \colon \
e_{1}\cdot e_{1}=e_{1},\quad e_{2}\cdot e_{2}=e_{2}, \quad
e_{1}\ast e_{1}=( \gamma +\beta -\alpha ) e_{1}-\beta e_{2},\\
\hphantom{\mathcal{C}_{38}^{\alpha ,\beta ,\gamma } \colon} \
e_{1}\ast e_{2}=\alpha e_{1}+\beta e_{2}, \quad
e_{2}\ast e_{1}=\alpha
e_{1}+\beta e_{2},\quad e_{2}\ast e_{2}=-\alpha e_{1}+\gamma e_{2}.
\end{gather*}
The algebras \smash{$\mathcal{C}_{38}^{\alpha ,\beta ,\gamma }$} and \smash{$\mathcal{C}%
_{38}^{\alpha ^{\prime },\beta ^{\prime },\gamma ^{\prime }}$} are isomorphic
if and only if
$
(\alpha, \beta, \gamma)=\bigl(\alpha ^{\prime },\beta ^{\prime },
\gamma ^{\prime }\bigr)$ or $(\alpha, \beta, \gamma)=\bigl(\beta ^{\prime },\alpha ^{\prime
}, -\alpha ^{\prime }+\beta ^{\prime }+\gamma ^{\prime }\bigr)$.

Now, assume that the compatible pre-Lie algebra structures defined on $\mathcal{C%
}_{07}$ has the second form. Then we obtain the algebras
\begin{gather*}
\mathcal{C}_{39}^{\alpha ,\beta \neq \alpha } \colon \
e_{1}\cdot e_{1}=e_{1}, \quad e_{2}\cdot e_{2}=e_{2}, \quad
e_{1}\ast e_{1}=\alpha e_{1}, \quad e_{2}\ast e_{2}=\beta e_{2}.
\end{gather*}
The algebras \smash{$\mathcal{C}_{39}^{\alpha ,\beta }$} and \smash{$\mathcal{C}_{39}^{\alpha
^{\prime },\beta ^{\prime }}$} are isomorphic if and only if
$(\alpha,\beta)= \bigl(\alpha^{\prime },\beta ^{\prime }\bigr)$ or
$(\alpha,\beta)= \bigl(\beta^{\prime },\alpha ^{\prime }\bigr)$.

\subsubsection[Compatible pre-Lie algebras defined on C\_\{08\}]{Compatible pre-Lie algebras defined on $\boldsymbol{\mathcal{C}_{08}}$}

From the computation of $\mathrm{Z}_{\mathrm{CPL}}^{2}(\mathcal{C}_{08},%
\mathcal{C}_{08})$, the compatible pre-Lie algebra structures defined on~$%
\mathcal{C}_{08}$ are of the following forms:
\begin{enumerate}\itemsep=0pt
\item[(1)] $e_{1}\cdot e_{1}=e_{1}$, $e_{1}\cdot e_{2}=2e_{2}$, $e_{2}\cdot e_{1}=%
\frac{1}{2}e_{1}+e_{2}$, $e_{2}\cdot e_{2}=e_{2}$,
$e_{1}\ast e_{1}=\alpha _{1}e_{1}+\alpha _{2}e_{2}$, $e_{1}\ast
e_{2}=2\alpha _{1}e_{2}$, $e_{2}\ast e_{1}=\alpha _{3}e_{1}+\alpha
_{1}e_{2}$, $e_{2}\ast e_{2}=2\alpha _{3}e_{2}$,

\item[(2)] $e_{1}\cdot e_{1}=e_{1}$, $e_{1}\cdot e_{2}=2e_{2}$, $e_{2}\cdot e_{1}=
\frac{1}{2}e_{1}+e_{2}$, $e_{2}\cdot e_{2}=e_{2}$,
$e_{1}\ast e_{1}=\alpha _{4}e_{1}$, $e_{1}\ast e_{2}=\alpha
_{5}e_{1}+2\alpha _{4}e_{2}$,
$e_{2}\ast e_{1}=\bigl( 2\alpha _{5}+\frac{1%
}{2}\alpha _{4}\bigr) e_{1}+\alpha _{4}e_{2}$, $e_{2}\ast e_{2}=\frac{1}{2
}\alpha _{5}e_{1}+( \alpha _{4}+\alpha _{5}) e_{2}$.
\end{enumerate}

We may assume $\alpha _{5}\neq 0$. Then we obtain the following compatible
pre-Lie algebras:
\begin{gather*}
\mathcal{C}_{40}^{\alpha ,\beta ,\gamma } \colon \
e_{1}\cdot e_{1}=e_{1}, \quad e_{1}\cdot e_{2}=2e_{2} , \quad e_{2}\cdot e_{1}=%
\tfrac{1}{2}e_{1}+e_{2},\quad e_{2}\cdot e_{2}=e_{2} , \\
\hphantom{\mathcal{C}_{40}^{\alpha ,\beta ,\gamma } \colon} \
e_{1}\ast e_{1}=\alpha e_{1}+\beta e_{2} , \quad e_{1}\ast e_{2}=2\alpha e_{2},
 \quad e_{2}\ast e_{1}=\gamma e_{1}+\alpha e_{2} , \quad e_{2}\ast e_{2}=2\gamma
e_{2},\\
\mathcal{C}_{41}^{\alpha ,\beta \neq 0} \colon \
 e_{1}\cdot e_{1}=e_{1} , \quad e_{1}\cdot e_{2}=2e_{2} , \quad e_{2}\cdot e_{1}=%
\tfrac{1}{2}e_{1}+e_{2} , \quad e_{2}\cdot e_{2}=e_{2} , \\
\hphantom{\mathcal{C}_{41}^{\alpha ,\beta \neq 0} \colon} \
 e_{1}\ast e_{1}=\alpha e_{1} ,\quad
 e_{1}\ast e_{2}=\beta e_{1}+2\alpha e_{2},
 \\
\hphantom{\mathcal{C}_{41}^{\alpha ,\beta \neq 0} \colon} \
 e_{2}\ast e_{1}=\bigl( 2\beta +\tfrac{1}{2}\alpha \bigr) e_{1}+\alpha
e_{2} , \quad e_{2}\ast e_{2}=\tfrac{1}{2}\beta e_{1}+( \alpha +\beta
) e_{2}.
\end{gather*}
Moreover, the algebras \smash{$\mathcal{C}_{40}^{\alpha ,\beta ,\gamma }$} and \smash{$\mathcal{C}%
_{40}^{\alpha ^{\prime },\beta ^{\prime },\gamma ^{\prime }}$} are isomorphic
if and only if
$\bigl( \alpha ^{\prime },\beta ^{\prime },\gamma ^{\prime
}\bigr) =( \alpha ,\beta ,\gamma )$ or~${\bigl( \alpha ^{\prime
},\beta ^{\prime },\gamma ^{\prime }\bigr) =( 4\gamma -\alpha ,\beta
-8\alpha +16\gamma ,\gamma )}$.
Also, the algebras \smash{$\mathcal{C}%
_{41}^{\alpha ,\beta }$} and \smash{$\mathcal{C}_{41}^{\alpha ^{\prime },\beta ^{\prime }}$}
are isomorphic if and only if
$
\bigl( \alpha ^{\prime },\beta ^{\prime
}\bigr) =( \alpha ,\beta ) $ or $\bigl( \alpha ^{\prime },\beta
^{\prime }\bigr) =( \alpha +4\beta ,-\beta )$.

\subsection[The algebraic classification of compatible commutative associative algebras]{The algebraic classification of compatible commutative\\ associative algebras}
It is easy to see that each commutative pre-Lie algebra is associative.
Hence, we can choose only commutative compatible pre-Lie algebras from Theorem \ref{compre2}.

\begin{Theorem}\label{comasscom2}
Let $\mathcal{C}$ be a nonzero $2$-dimensional compatible commutative associative algebra.
Then $\mathcal{C}$ is isomorphic to one and only one and only one of the following
algebras:
\begin{gather*}
 \mathcal{C}_{03} \colon\ e_{1}\ast e_{1} =e_{2}, \\
 \mathcal{C}_{05}^{0} \colon\ e_{1}\ast e_{1} =e_{1}, \\
 \mathcal{C}_{06}^{1} \colon\ e_{1} \ast e_{1} =e_{1}, \quad e_{1} \ast e_{2} = e_{2}, \quad e_{2} \ast e_{1} =e_{2}, \\
 \mathcal{C}_{07} \colon\ e_{1}\ast e_{1} =e_{1},\quad e_{2}\ast e_{2}=e_{2}, \\
 \mathcal{C}_{13}^{1 } \colon\
 e_{1}\cdot e_{1}=e_{2} , \quad
 e_{1}\ast e_{1}= e_{1} , \quad e_{1}\ast e_{2}=e_{2} , \quad e_{2}\ast
e_{1}= e_{2},
\\
 \mathcal{C}_{15}^{\alpha } \colon\
 e_{1}\cdot e_{1}=e_{2} , \quad
 e_{1}\ast e_{1}=\alpha e_{2}, \\
 \mathcal{C}_{16}^{0 } \colon\
 e_{1}\cdot e_{1}=e_{2} , \quad
 e_{1}\ast e_{1}=e_{1},
\\
 \mathcal{C}_{18}^{\alpha } \colon\
 e_{1}\cdot e_{1}=e_{2} , \quad
 e_{1}\ast e_{1}=\alpha e_{2} , \quad e_{1}\ast e_{2}=e_{1} , \quad e_{2}\ast
e_{1}=e_{1} , \quad e_{2}\ast e_{2}=e_{2},
\\
 \mathcal{C}_{24}^{0 ,\beta ,0 } \colon\
 e_{1}\cdot e_{1}=e_{1} , \quad
 e_{1}\ast e_{1}=\beta e_{1}+e_{2},\\
 \mathcal{C}_{25}^{0 ,\beta ,0 } \colon\
 e_{1}\cdot e_{1}=e_{1} , \quad
 e_{1}\ast e_{1}=\beta e_{1},
\\
 \mathcal{C}_{29}^{\alpha } \colon\
 e_{1}\cdot e_{1}=e_{1} , \quad
 e_{1}\ast e_{1}=\alpha e_{1} , \quad e_{2}\ast e_{2}=e_{2},
 \\
 \mathcal{C}_{30}^{\alpha ,\beta } \colon\
 e_{1}\cdot e_{1}=e_{1} , \quad
 e_{1}\ast e_{1}=\alpha e_{1}+\beta e_{2} , \quad e_{1}\ast e_{2}=e_{1} , \quad
e_{2}\ast e_{1}=e_{1} , \quad e_{2}\ast e_{2}=e_{2},
\\
 \mathcal{C}_{31}^{1 ,\beta ,\beta } \colon\
 e_{1}\cdot e_{1}=e_{1} , \quad e_{1}\cdot e_{2}= e_{2} , \quad e_{2}\cdot
e_{1}=e_{2} , \\
\hphantom{\mathcal{C}_{31}^{1 ,\beta ,\beta } \colon} \
 e_{1}\ast e_{1}=\beta e_{1}+e_{2} ,\quad e_{1}\ast e_{2}=\beta e_{2} , \quad
e_{2}\ast e_{1}=\beta e_{2},
\\
 \mathcal{C}_{32}^{1 ,\beta ,\beta } \colon\
 e_{1}\cdot e_{1}=e_{1} , \quad e_{1}\cdot e_{2}= e_{2} , \quad e_{2}\cdot
e_{1}=e_{2} , \\
\hphantom{\mathcal{C}_{32}^{1 ,\beta ,\beta } \colon} \
 e_{1}\ast e_{1}=\beta e_{1} , \quad e_{1}\ast e_{2}=\beta e_{2} , \quad
e_{2}\ast e_{1}=\beta e_{2},
 \\
 \mathcal{C}_{34}^{\alpha ,\beta } \colon\
 e_{1}\cdot e_{1}=e_{1} , \quad e_{1}\cdot e_{2}=e_{2} , \quad e_{2}\cdot
e_{1}=e_{2} , \\
\hphantom{\mathcal{C}_{34}^{\alpha ,\beta } \colon} \
 e_{1}\ast e_{1}=\alpha e_{1} , \quad e_{1}\ast e_{2}=\alpha e_{2} , \quad
e_{2}\ast e_{1}=\alpha e_{2} , \quad e_{2}\ast e_{2}=e_{1}+\beta e_{2},
 \\
 \mathcal{C}_{35}^{\alpha } \colon\
 e_{1}\cdot e_{1}=e_{1} , \quad e_{1}\cdot e_{2}=e_{2} , \quad e_{2}\cdot
e_{1}=e_{2} , \\
\hphantom{\mathcal{C}_{35}^{\alpha } \colon} \
 e_{1}\ast e_{1}=\alpha e_{1} , \quad e_{1}\ast e_{2}=\alpha e_{2} , \quad
e_{2}\ast e_{1}=\alpha e_{2} , \quad e_{2}\ast e_{2}=e_{2},
 \\
 \mathcal{C}_{38}^{\alpha ,\beta ,\gamma } \colon\
 e_{1}\cdot e_{1}=e_{1} , \quad e_{2}\cdot e_{2}=e_{2} , \quad
 e_{1}\ast e_{1}=( \gamma +\beta -\alpha ) e_{1}-\beta e_{2} ,\\
 \hphantom{\mathcal{C}_{38}^{\alpha ,\beta ,\gamma } \colon} \
e_{1}\ast e_{2}=\alpha e_{1}+\beta e_{2} , \quad
 e_{2}\ast e_{1}=\alpha
e_{1}+\beta e_{2} , \quad e_{2}\ast e_{2}=-\alpha e_{1}+\gamma e_{2},
 \\
 \mathcal{C}_{39}^{\alpha ,\beta \neq \alpha } \colon\
 e_{1}\cdot e_{1}=e_{1} ,\quad e_{2}\cdot e_{2}=e_{2} , \quad
 e_{1}\ast e_{1}=\alpha e_{1} , \quad e_{2}\ast e_{2}=\beta e_{2}.
\end{gather*}

All algebras are non-isomorphic, except
\smash{$\mathcal{C}_{38}^{\alpha ,\beta ,\gamma } \cong \mathcal{C}_{38}^{\beta ,\alpha , -\alpha +\beta +\gamma }$},
\smash{$\mathcal{C}_{39}^{\alpha ,\beta } \cong \mathcal{C}_{39}^{\beta,\alpha}$}.
\end{Theorem}

\subsection{The algebraic classification of compatible associative algebras}

\begin{Definition}
A compatible associative algebra is a vector space ${\bf A}$ equipped with
two multiplications: $\cdot$ and another multiplication $\ast$,
such that, $({\bf A}, \cdot)$, $({\bf A}, \ast)$ and
$({\bf A}, \cdot+ \ast)$ are associative algebras.
These two operations are required to satisfy the following identities:
\begin{gather*}
( x\cdot y) \cdot z=x\cdot ( y\cdot
z),\quad
( x\ast y) \ast z=x\ast ( y\ast z),\\
(x\ast y)\cdot z+(x\cdot y)\ast z=
x\ast (y\cdot z)+x\cdot (y\ast z).
 \end{gather*}
 \end{Definition}

\begin{Theorem}
\label{comass2}
Let $\mathcal{C}$ be a nonzero $2$-dimensional compatible associative algebra.
Then $\mathcal{C}$ is isomorphic to one and only one compatible commutative associative algebra listed in Theorem~{\rm\ref{comasscom2}} or one of the following algebras:
\begin{gather*}
\mathcal{C}_{05}^{1}\colon\ e_{1}\ast e_{1} =e_{1},\quad e_{1}\ast e_{2} = e_{2},\\
\mathcal{C}_{06}^{0}\colon\ e_{1} \ast e_{1} =e_{1}, \quad e_{2} \ast e_{1} =e_{2},\\
\mathcal{C}_{25}^{1 ,\beta , \beta }\colon\
e_{1}\cdot e_{1}=e_{1}, \quad e_{1}\cdot e_{2}= e_{2}, \quad
e_{1}\ast e_{1}=\beta e_{1}, \quad e_{1}\ast e_{2}=\beta e_{2},
\\
\mathcal{C}_{28}\colon\
e_{1}\cdot e_{1}=e_{1}, \quad e_{1}\cdot e_{2}=e_{2}, \quad
e_{2}\ast e_{1}=e_{1}, \quad e_{2}\ast e_{2}=e_{2},
\\
\mathcal{C}_{32}^{0 ,\beta ,0 }\colon\
e_{1}\cdot e_{1}=e_{1}, \quad e_{2}\cdot
e_{1}=e_{2}, \quad
e_{1}\ast e_{1}=\beta e_{1}, \quad
e_{2}\ast e_{1}=\beta e_{2},
\\
\mathcal{C}_{33}\colon\
e_{1}\cdot e_{1}=e_{1}, \quad e_{2}\cdot e_{1}=e_{2}, \quad
e_{1}\ast e_{2}=e_{1}, \quad e_{2}\ast e_{2}=e_{2}.
\end{gather*}
\end{Theorem}

\subsection{The algebraic classification of compatible Novikov algebras}

\begin{Definition}
A compatible Novikov algebra is a vector space ${\bf A}$ equipped with
two multiplications: $\cdot$ and another multiplication $\ast$,
such that, $({\bf A}, \cdot)$, $({\bf A}, \ast)$ and
$({\bf A}, \cdot+ \ast)$ are Novikov algebras.
These two operations are required to satisfy the following identities:
\begin{gather*}
 ( x\cdot y) \cdot z-x\cdot ( y\cdot
z) = ( y\cdot x) \cdot z -y\cdot ( x\cdot z) ,\quad
 (x \cdot y) \cdot z = (x \cdot z)\cdot y, \\
 ( x\ast y) \ast z-x\ast ( y\ast
z) = ( y\ast x) \ast z -y\ast ( x\ast z) ,\quad
 (x \ast y) \ast z = (x \ast z)\ast y, \\
(x\ast y)\cdot z-x\ast (y\cdot z)+(x\cdot y)\ast z-x\cdot (y\ast z) \\
\qquad =
(y\ast x)\cdot z-y\ast (x\cdot z)+(y\cdot x)\ast z-y\cdot (x\ast z) ,\\
 (x \ast y) \cdot z +(x \cdot y) \ast z = (x \ast z)\cdot y+(x \cdot z)\ast y.
 \end{gather*}
 \end{Definition}

\begin{Theorem}\label{comnov2}
Let $\mathcal{C}$ be a nonzero $2$-dimensional compatible Novikov algebra.
Then $\mathcal{C}$ is isomorphic to one and only one compatible commutative associative algebra listed in Theorem~{\rm\ref{comasscom2}} or one of the following algebras:
\begin{gather*}
 \mathcal{C}_{01} \colon\ e_{1} \ast e_{1} = e_{1} + e_{2},\quad e_{2} \ast e_{1} =e_{2}, \\
 \mathcal{C}_{04} \colon\ e_{2} \ast e_{1} =e_{1}, \\
 \mathcal{C}_{06}^{\alpha\neq 1} \colon\ e_{1} \ast e_{1} =e_{1}, \quad e_{1} \ast e_{2} = \alpha e_{2}, \quad e_{2} \ast e_{1} =e_{2}, \\
 \mathcal{C}_{09}^{\alpha ,\beta \neq 0} \colon\
 e_{1}\cdot e_{1}=e_{1}+e_{2} , \quad e_{2}\cdot e_{1}=e_{2} , \quad
 e_{1}\ast e_{1}=\alpha e_{1} , \quad e_{1}\ast e_{2}=\beta e_{2} , e_{2}\ast e_{1}=\alpha e_{2},
\\
 \mathcal{C}_{10}^{\alpha ,\beta } \colon\
 e_{1}\cdot e_{1}=e_{1}+e_{2} ,\quad e_{2}\cdot e_{1}=e_{2} , \quad
 e_{1}\ast e_{1}=\alpha e_{1}+\beta e_{2} , \quad e_{2}\ast e_{1}=\alpha e_{2},
\\
 \mathcal{C}_{13}^{\alpha\neq1 } \colon\
 e_{1}\cdot e_{1}=e_{2} , \quad
 e_{1}\ast e_{1}=\alpha e_{1} , \quad e_{1}\ast e_{2}=e_{2} , \quad e_{2}\ast
e_{1}=\alpha e_{2},
\\
 \mathcal{C}_{14}^{\alpha } \colon\
 e_{1}\cdot e_{1}=e_{2} , \quad
 e_{1}\ast e_{1}=e_{1}+\alpha e_{2} ,\quad e_{2}\ast e_{1}=e_{2},
\\
 \mathcal{C}_{20}^{\alpha ,0 } \colon\
 e_{2}\cdot e_{1}=e_{1} , \quad
 e_{2}\ast e_{1}=\alpha e_{1} , \quad e_{2}\ast e_{2}=e_{1},
\\
 \mathcal{C}_{21}^{\alpha ,0 } \colon\
 e_{2}\cdot e_{1}=e_{1} , \quad
 e_{2}\ast e_{1}=\alpha e_{1},
\\
 \mathcal{C}_{22}^{\alpha \neq 0,\beta } \colon\
 e_{2}\cdot e_{1}=e_{1} , \quad
 e_{1}\ast e_{2}=\alpha e_{1} , \quad e_{2}\ast e_{1}=\beta e_{1} , \quad
e_{2}\ast e_{2}=e_{1}+\alpha e_{2},
\\
 \mathcal{C}_{23}^{\alpha \neq 0,\beta } \colon\
 e_{2}\cdot e_{1}=e_{1} , \quad
 e_{1}\ast e_{2}=\alpha e_{1} , \quad e_{2}\ast e_{1}=\beta e_{1} ,\quad
e_{2}\ast e_{2}=\alpha e_{2},
\\
 \mathcal{C}_{31}^{(\alpha ,\beta ,\gamma) \neq (1,\beta,\beta) } \colon\
 e_{1}\cdot e_{1}=e_{1} , \quad e_{1}\cdot e_{2}=\alpha e_{2} , \quad e_{2}\cdot
e_{1}=e_{2} , \\
\hphantom{\mathcal{C}_{31}^{(\alpha ,\beta ,\gamma) \neq (1,\beta,\beta) } \colon} \
 e_{1}\ast e_{1}=\beta e_{1}+e_{2} , \quad e_{1}\ast e_{2}=\gamma e_{2} , \quad
e_{2}\ast e_{1}=\beta e_{2},
\\
 \mathcal{C}_{32}^{(\alpha ,\beta ,\gamma) \neq (1,\beta,\beta) } \colon\
e_{1}\cdot e_{1}=e_{1} , \quad e_{1}\cdot e_{2}=\alpha e_{2} , \quad e_{2}\cdot
e_{1}=e_{2} , \\
\hphantom{\mathcal{C}_{32}^{(\alpha ,\beta ,\gamma) \neq (1,\beta,\beta) } \colon} \
 e_{1}\ast e_{1}=\beta e_{1} , \quad e_{1}\ast e_{2}=\gamma e_{2} , \quad
e_{2}\ast e_{1}=\beta e_{2},
\\
 \mathcal{C}_{33} \colon\
 e_{1}\cdot e_{1}=e_{1} , \quad e_{2}\cdot e_{1}=e_{2} , \quad
 e_{1}\ast e_{2}=e_{1} , \quad e_{2}\ast e_{2}=e_{2}. %
\end{gather*}

All algebras are non-isomorphic, except
\smash{$\mathcal{C}_{31}^{0 ,\beta ,\gamma } \cong \mathcal{C}_{32}^{0 , \beta ,\gamma }$}.
\end{Theorem}

\section{The geometric classification of compatible algebras}

 \subsection{Degenerations and the geometric classification of algebras}
Let us introduce the techniques used to obtain the geometric classification for an arbitrary variety of compatible $\Omega$-algebras. Given a complex vector space ${\mathbb V}$ of dimension $n$, the set of bilinear maps
\[
\textrm{Bil}({\mathbb V} \times {\mathbb V}, {\mathbb V}) \cong \textrm{Hom}\bigl({\mathbb V} ^{\otimes2}, {\mathbb V}\bigr)\cong ({\mathbb V}^*)^{\otimes2} \otimes {\mathbb V}
\]
 is a vector space of dimension $n^3$. The set of pairs of bilinear maps 
\[
\textrm{Bil}({\mathbb V} \times {\mathbb V}, {\mathbb V}) \oplus \textrm{Bil}({\mathbb V}\times {\mathbb V}, {\mathbb V}) \cong ({\mathbb V}^*)^{\otimes2} \otimes {\mathbb V} \oplus ({\mathbb V}^*)^{\otimes2} \otimes {\mathbb V},
\]
 which is a vector space of dimension $2n^3$. This vector space has the structure of the affine space~$\mathbb{C}^{2n^3}$ in the following sense:
fixed a basis $e_1, \ldots, e_n$ of ${\mathbb V}$, then any pair with multiplication~$(\mu, \mu')$, is
determined by some parameters $c_{ij}^k, c_{ij}'^k \in \mathbb{C}$, called {structural constants}, such that
\[
\mu(e_i, e_j) = \sum_{p=1}^n c_{ij}^k e_k \quad \textrm{and}\quad \mu'(e_i, e_j) = \sum_{p=1}^n c_{ij}'^k e_k,
\]
which corresponds to a point in the affine space \smash{$\mathbb{C}^{2n^3}$}. Then a set of bilinear 
pairs $\mathcal S$ corresponds to an algebraic variety, i.e., a Zariski closed set, if there are some polynomial equations in variables $c_{ij}^k$, $c_{ij}'^k$ with zero locus equal to the set of structural constants of the bilinear pairs in~$\mathcal S$. Given the identities defining a particular class of compatible $\Omega$-algebras, we can obtain a set of polynomial equations in variables $c_{ij}^k$, $c_{ij}'^k$. This class of $n$-dimensional compatible $\Omega$-algebras is a variety. Denote it by $\mathcal{T}_{n}$. 
Now, consider the following action of $\rm{GL}({\mathbb V})$ on ${\mathcal T}_{n}$
\[
(g*(\mu, \mu'))(x,y) := \bigl(g \mu \bigl(g^{-1} x, g^{-1} y\bigr), g \mu' \bigl(g^{-1} x, g^{-1} y\bigr)\bigr)
\]
for $g\in\rm{GL}({\mathbb V})$, $(\mu, \mu')\in \mathcal{T}_{n}$ and for any $x, y \in {\mathbb V}$. Observe that the $\textrm{GL}({\mathbb V})$-orbit of $(\mu, \mu')$, denoted ${\mathcal O}((\mu, \mu'))$, contains all the structural constants of the bilinear pairs isomorphic to the compatible $\Omega$-algebras with structural constants $(\mu, \mu')$.

A geometric classification of a variety of algebras consists of describing the irreducible components of the variety. Recall that any affine variety can be represented as a finite union of its irreducible components uniquely.
Note that describing the irreducible components of ${\mathcal{T}_{n}}$ gives us the rigid
algebras of the variety, which are those bilinear pairs with an open $\textrm{GL}(\mathbb V)$-orbit. This 
is due to the fact that a bilinear pair is rigid in a variety if and only if the closure of its orbit is an irreducible component of the variety. 
For this reason, the following notion is convenient. Denote by $\overline{{\mathcal O}((\mu, \mu'))}$ the closure of the orbit of $(\mu, \mu')\in{\mathcal{T}_{n}}$.

\begin{Definition}
Let ${\rm T} $ and ${\rm T}'$ be two $n$-dimensional compatible $\Omega$-algebras of a fixed class corresponding to the variety $\mathcal{T}_{n}$ and $(\mu, \mu'), (\lambda,\lambda') \in \mathcal{T}_{n}$ be their representatives in the affine space, respectively. The algebra ${\rm T}$ is said to {degenerate} to ${\rm T}'$, and we write ${\rm T} \to {\rm T} '$, if $(\lambda,\lambda')\in\overline{{\mathcal O}((\mu, \mu'))}$. If ${\rm T} \not\cong {\rm T}'$, then we call it a {proper degeneration}.
Conversely, if $(\lambda,\lambda')\not\in\overline{{\mathcal O}((\mu, \mu'))}$ then we say that ${{\rm T} }$ does not degenerate to ${{\rm T} }'$
and we write ${{\rm T} }\not\to {{\rm T} }'$.
\end{Definition}

Furthermore, for a parametric family of algebras, we have the following notion.

\begin{Definition}
Let ${{\rm T} }(*)=\{{{\rm T} }(\alpha)\colon {\alpha\in I}\}$ be a family of $n$-dimensional compatible $\Omega$-algebras of a fixed class corresponding to ${\mathcal{T} }_n$ and let ${{\rm T} }'$ be another algebra. Suppose that ${{\rm T} }(\alpha)$ is represented by the structure $(\mu(\alpha),\mu'(\alpha))\in{\mathcal{T} }_n$ for $\alpha\in I$ and ${{\rm T} }'$ is represented by the structure~${(\lambda, \lambda')\in{\mathcal{T} }_n}$. We say that the family ${{\rm T} }(*)$ {degenerates} to ${{\rm T} }'$, and write ${{\rm T} }(*)\to {{\rm T} }'$, if~\smash{$(\lambda,\lambda')\in\overline{\{{\mathcal O}((\mu(\alpha),\mu'(\alpha)))\}_{\alpha\in I}}$}.
Conversely, if~\smash{$(\lambda,\lambda')\not\in\overline{\{{\mathcal O}((\mu(\alpha),\mu'(\alpha)))\}_{\alpha\in I}}$} then we call it a~{non-degeneration}, and we write ${{\rm T} }(*)\not\to {{\rm T} }'$.
\end{Definition}

Observe that ${\rm T}'$ corresponds to an irreducible component of $\mathcal{T}_n$ (more precisely, $\overline{{\rm T}'}$ is an irreducible component) if and only if ${{\rm T} }\not\to {{\rm T} }'$ for any $n$-dimensional compatible $\Omega$-algebra~${\rm T}$ and~${{{\rm T}(*) }\not\to {{\rm T} }'}$ for any parametric family of $n$-dimensional compatible $\Omega$-algebras ${\rm T}(*)$. To prove a~particular algebra corresponds to an irreducible component, we will use the next ideas.
Firstly, since $\operatorname{dim}{\mathcal O}((\mu, \mu')) = n^2 - \operatorname{dim}\mathfrak{Der}({\rm T})$, then if $ {\rm T} \to {\rm T} '$ and ${\rm T} \not\cong {\rm T} '$, we have that $\operatorname{dim}\mathfrak{Der}( {\rm T} )<\operatorname{dim}\mathfrak{Der}( {\rm T} ')$, where $\mathfrak{Der}( {\rm T} )$ denotes the Lie algebra of derivations of ${\rm T} $.
Secondly, to prove degenerations, let ${{\rm T} }$ and ${{\rm T} }'$ be two compatible $\Omega$-algebras represented by the structures~$(\mu, \mu')$ and $(\lambda, \lambda')$ from ${{\mathcal T} }_n$, respectively. Let \smash{$c_{ij}^k$}, \smash{$c_{ij}'^k$} be the structure constants of $(\lambda, \lambda')$ in a~basis $e_1,\dots, e_n$ of ${\mathbb V}$. If there exist $n^2$ maps \smash{$a_i^j(t)\colon \mathbb{C}^*\to \mathbb{C}$} such that \smash{$E_i(t)=\sum_{j=1}^na_i^j(t)e_j$} ($1\leq i \leq n$) form a basis of ${\mathbb V}$ for any $t\in\mathbb{C}^*$ and the structure constants $c_{ij}^k(t)$, $c_{ij}'^k(t)$ of $(\mu, \mu')$ in the basis $E_1(t),\dots, E_n(t)$ satisfy \smash{$\lim_{t\to 0}c_{ij}^k(t)=c_{ij}^k$} and \smash{$\lim_{t\to 0}c_{ij}'^k(t)=c_{ij}'^k$}, then ${{\rm T} }\to {{\rm T} }'$. In this case, $E_1(t),\dots, E_n(t)$ is called a parametrized basis for ${{\rm T} }\to {{\rm T} }'$.
In case of $E_1^t, E_2^t, \ldots, E_n^t$ is a~parametric basis for ${\bf A}\to {\bf B}$, it will be denoted by
\[
{\bf A}\xrightarrow{(E_1^t, E_2^t, \ldots, E_n^t)} {\bf B}.
\]

Thirdly, to prove non-degenerations we may use a remark that follows from this lemma, see~\cite{afm}.

\begin{Lemma} 
Consider two compatible $\Omega$-algebras ${\rm T}$ and ${\rm T}'$. Suppose ${\rm T} \to {\rm T}'$. Let C be a Zariski closed in ${\mathcal T}_n$ that is stable by the action of the invertible upper $($lower$)$ triangular matrices. Then if there is a representation $(\mu, \mu')$ of ${\rm T}$ in C, then there is a representation $(\lambda, \lambda')$ of ${\rm T}'$ in C.
\end{Lemma}

In order to apply this lemma, we will give the explicit definition of the appropriate stable Zariski closed $C$ in terms of the variables $c_{ij}^k$, $c_{ij}'^k$ in each case.

\begin{Remark}\label{redbil}
Moreover, let ${{\rm T} }$ and ${{\rm T} }'$ be two compatible $\Omega$-algebras represented by the structures $(\mu, \mu')$ and $(\lambda, \lambda')$ from ${\mathcal{T}_{n}}$. Suppose ${\rm T}\to{\rm T}'$. Then if $\mu$, $\mu'$, $\lambda$, $\lambda'$ represents algebras ${\rm T}_{0}$, ${\rm T}_{1}$, ${\rm T}'_{0}$, ${\rm T}'_{1}$ in the affine space $\mathbb{C}^{n^3}$ of algebras with a single multiplication, respectively, we have~${{\rm T}_{0}\to {\rm T}'_{0}}$ and $ {\rm T}_{1}\to {\rm T}'_{1}$.
So, for example, $(0, \mu)$ can not degenerate in $(\lambda, 0)$ unless $\lambda=0$.
\end{Remark}

Fourthly, to prove ${{\rm T} }(*)\to {{\rm T} }'$, suppose that ${{\rm T} }(\alpha)$ is represented by the structure $(\mu(\alpha),\mu'(\alpha))\allowbreak\in{\mathcal{T} }_n$ for $\alpha\in I$ and ${{\rm T} }'$ is represented by the structure $(\lambda, \lambda')\in{\mathcal{T} }_n$. Let $c_{ij}^k$, $c_{ij}'^k$ be the structure constants of $(\lambda, \lambda')$ in a basis $e_1,\dots, e_n$ of ${\mathbb V}$. If there is a pair of maps \smash{$\bigl(f, \bigl(a_i^j\bigr)\bigr)$}, where $f\colon \mathbb{C}^*\to I$ and $a_i^j\colon\mathbb{C}^*\to \mathbb{C}$ are such that \smash{$E_i(t)=\sum_{j=1}^na_i^j(t)e_j$} ($1\le i\le n$) form a basis of ${\mathbb V}$ for any~${t\in\mathbb{C}^*}$ and the structure constants $c_{ij}^k(t)$, $c_{ij}'^k(t)$ of $(\mu(f(t)),\mu'(f(t)))$ in the basis $E_1(t),\dots, E_n(t)$ satisfy~\smash{$\lim_{t\to 0}c_{ij}^k(t)=c_{ij}^k$} and \smash{$\lim_{t\to 0}c_{ij}'^k(t)=c_{ij}'^k$}, then ${{\rm T} }(*)\to {{\rm T} }'$. In this case, $E_1(t),\dots, E_n(t)$ and $f(t)$ are called a parametrized basis and a parametrized index for ${{\rm T} }(*)\to {{\rm T} }'$, respectively.
Fifthly, to prove ${{\rm T} }(*)\not \to {{\rm T} }'$, we may use an analogous of Remark \ref{redbil} for parametric families that follows from Lemma \ref{main2}.

\begin{Lemma}\label{main2}
Consider a family of compatible $\Omega$-algebras ${\rm T}(*)$ and a compatible $\Omega$-algebra ${\rm T}'$. Suppose ${\rm T}(*) \to {\rm T}'$. Let C be a Zariski closed in $\mathcal{T}_n$ that is stable by the action of the invertible upper $($lower$)$ triangular matrices. Then if there is a representation $(\mu(\alpha), \mu'(\alpha))$ of ${\rm T}(\alpha)$ in C for every $\alpha\in I$, then there is a representation $(\lambda, \lambda')$ of ${\rm T}'$ in C.
\end{Lemma}

Finally, the following remark simplifies the geometric problem.

\begin{Remark}\label{remrem}
 Let $(\mu, \mu')$ and $(\lambda, \lambda')$ represent two compatible $\Omega$-algebras. Suppose $(\lambda, 0)\not\in\overline{{\mathcal O}((\mu, 0))}$, (resp.\ $(0, \lambda')\not\in\overline{{\mathcal O}((0, \mu'))}$), then $(\lambda, \lambda')\not\in\overline{{\mathcal O}((\mu, \mu'))}$.
 As we construct the classification of a given class of compatible $\Omega$-algebras from a certain class of algebras with a single multiplication that remains unchanged, this remark becomes very useful.
\end{Remark}

\subsection[The geometric classification of compatible commutative associative algebras]{The geometric classification of compatible commutative\\ associative algebras}
The main result of the present section is the following theorem.

\begin{Theorem}\label{geo2}
The variety of complex $2$-dimensional compatible commutative associative algebras has
dimension $7$ and it has $2$ irreducible components defined by
\smash{$ \overline{\mathcal{O}\bigl( \mathcal{C}_{39}^{\alpha,\beta}\bigr)}$} and
\smash{$ \overline{\mathcal{O}\bigl( \mathcal{C}_{38}^{\alpha,\beta, \gamma}\bigr)}$}.
In particular, there are no rigid algebras in this variety.
\end{Theorem}

\begin{proof}
 Thanks to Theorem \ref{comasscom2}, we have the algebraic classification of
 $2$-dimensional compatible commutative associative algebras.
After carefully checking the dimensions of orbit closures of the more important for us algebras, we have
$\dim \mathcal{O}\bigl(\mathcal{C}_{38}^{\alpha,\beta,\gamma}\bigr) = 7 $ and $ \dim \mathcal{O}\bigl(\mathcal{C}_{39}^{\alpha,\beta}\bigr) = 6$.
\smash{$\mathcal{C}_{38}^{*} \not\to \mathcal{C}_{39}^{\alpha,\beta}$} due to the following relation:
\[
\mathcal R=\big\{
c_{22}^1=c_{12}^1=0,\,
c_{22}^2 c_{12}'^2 + c_{12}^2 c_{21}'^1 + c_{11}^1 c_{22}'^2 = c_{11}'^1 c_{22}^2 + c_{11}^1 c_{21}'^1 + c_{12}^2 c_{22}'^2\big\}.
\]

Thanks to \cite{KV}, we have
$\mathcal{C}_{07} \to \big\{\mathcal{C}_{03}, \mathcal{C}_{05}^0, \mathcal{C}_{06}^1\big\}.$ All necessary degenerations are given by
\begin{gather*}
 \mathcal{C}_{38}^{0, 0, t^{-1}} \xrightarrow{ (te_1, t e_2)} \mathcal{C}_{07} ,\quad
 \mathcal{C}_{39}^{-it^{-1}, it^{-1}} \xrightarrow{ (i te_1-i te_2, -t^2e_1-t^2e_2)} \mathcal{C}_{13}^{1} , \\
 \mathcal{C}_{38}^{0, 0, \alpha} \xrightarrow{ (i te_1-i te_2, -t^2e_1-t^2e_2)} \mathcal{C}_{15}^{\alpha} ,\quad
 \mathcal{C}_{38}^{\frac{i}{2t}, 0, \frac{i}{t}} \xrightarrow{ (i te_1-i te_2, -t^2e_1-t^2e_2)} \mathcal{C}_{16}^{0} , \\
 \mathcal{C}_{38}^{-\frac{1+\alpha t^2}{4 t^2}, -\frac{1+\alpha t^2}{4 t^2}, \frac{\alpha t^2-3}{4t^2}} \xrightarrow{ (i te_1-i t e_2, -t^2e_1-t^2e_2)} \mathcal{C}_{18}^{\alpha} ,\quad \mathcal{C}_{38}^{0, -t, \beta+t} \xrightarrow{ (e_1, te_2)} \mathcal{C}_{24}^{0,\beta,0} , \\
 \mathcal{C}_{38}^{0, 0, \beta} \xrightarrow{ (e_1, te_2)} \mathcal{C}_{25}^{0,\beta,0} ,\quad
 \mathcal{C}_{39}^{\alpha, t^{-1}} \xrightarrow{ (e_1, te_2)} \mathcal{C}_{29}^{\alpha} , \quad \mathcal{C}_{39}^{-t+\beta, \beta} \xrightarrow{ (e_1+e_2, te_2)} \mathcal{C}_{31}^{1,\beta,\beta} ,\\
 \mathcal{C}_{38}^{-\alpha-\beta t +t^{-1}, - \beta t , t^{-1}} \xrightarrow{ (e_1+e_2, te_2)} \mathcal{C}_{30}^{\alpha, \beta} ,\quad \mathcal{C}_{38}^{0, \beta, 0} \xrightarrow{ (e_1+e_2, te_2)} \mathcal{C}_{32}^{1, \beta, \beta} , \\
 \mathcal{C}_{38}^{-t^{-2}, \alpha -(1+\beta t)t^{-2}, (1+\beta t)t^{-2}} \xrightarrow{ (e_1+e_2, te_2)} \mathcal{C}_{34}^{\alpha, \beta} , \quad \mathcal{C}_{38}^{\frac{1+2\alpha t}{2t}, \alpha, 0} \xrightarrow{ (e_1+e_2, -te_1+te_2)} \mathcal{C}_{35}^{\alpha}.\tag*{\qed}
 \end{gather*}\renewcommand{\qed}{}
\end{proof}

\subsection{The geometric classification of compatible associative algebras}
The main result of the present section is the following theorem.

\begin{Theorem}\label{geo3}
The variety of complex $2$-dimensional compatible associative algebras has
dimension $7$ and it has $4$ irreducible components defined by
\smash{$ \overline{\mathcal{O}( \mathcal{C}_{28})}$},
\smash{$ \overline{\mathcal{O}( \mathcal{C}_{33})}$},
\smash{$ \overline{\mathcal{O}( \mathcal{C}_{39}^{\alpha,\beta})}$} and
\smash{$ \overline{\mathcal{O}( \mathcal{C}_{38}^{\alpha,\beta, \gamma})}$}.
In particular, there are two rigid algebras in this variety.

\end{Theorem}

\begin{proof}
 Thanks to Theorem \ref{comass2}, we have the algebraic classification of
 $2$-dimensional compatible associative algebras.
After carefully checking the dimensions of orbit closures of the more important for us algebras, we have
\begin{gather*}
\dim \mathcal{O}\bigl(\mathcal{C}_{38}^{\alpha,\beta,\gamma}\bigr)=7, \quad
\dim \mathcal{O}\bigl(\mathcal{C}_{39}^{\alpha,\beta}\bigr)=6, \quad
\dim \mathcal{O}(\mathcal{C}_{28})=\dim \mathcal{O}(\mathcal{C}_{34})=4.
\end{gather*}

Thanks to Theorem \ref{geo3}, we have that
\smash{$\mathcal{C}_{38}^{\alpha,\beta,\gamma}$} and \smash{$\mathcal{C}_{39}^{\alpha,\beta}$} give irreducible components in the variety of $2$-dimensional compatible commutative associative algebras.
The rest of the necessary degenerations are given by
\begin{gather*}
 \mathcal{C}_{28} \xrightarrow{ ( e_2, te_1)} \mathcal{C}_{05}^1, \quad
 \mathcal{C}_{33} \xrightarrow{ ( e_2, te_1)} \mathcal{C}_{06}^0, \quad
 \mathcal{C}_{28} \xrightarrow{ (e_1+\beta e_2, te_2)} \mathcal{C}_{25}^{1,\beta,\beta}, \quad
 \mathcal{C}_{33} \xrightarrow{ ( e_1+\beta e_2, te_2)} \mathcal{C}_{32}^{0,\beta,0}.\tag*{\qed} \end{gather*}\renewcommand{\qed}{}
\end{proof}

\subsection{The geometric classification of compatible Novikov algebras}

The main result of the present section is the following theorem.

\begin{Theorem}\label{geo1}
The variety of complex $2$-dimensional compatible Novikov algebras has
dimension~$7$ and it has $6$ irreducible components defined by
$ \overline{\mathcal{O}( \mathcal{C}_{33})}$,
\smash{$ \overline{\mathcal{O}\big( \mathcal{C}_{09}^{\alpha,\beta}\big)}$},
\smash{$ \overline{\mathcal{O}\big( \mathcal{C}_{22}^{\alpha,\beta}\big)}$},
\smash{$ \overline{\mathcal{O}\big( \mathcal{C}_{39}^{\alpha,\beta}\big)}$},
\smash{$ \overline{\mathcal{O}\big( \mathcal{C}_{31}^{\alpha,\beta, \gamma}\big)}$} and
\smash{$ \overline{\mathcal{O}\big( \mathcal{C}_{38}^{\alpha,\beta,\gamma}\big)}$}.
In particular, there is one rigid algebra in this variety.

\end{Theorem}

\begin{proof}
 Thanks to Theorem \ref{comnov2}, we have the algebraic classification of
 $2$-dimensional compatible Novikov algebras.
After carefully checking the dimensions of orbit closures of the more important for us algebras, we have
\begin{gather*}
\dim \mathcal{O}\bigl(\mathcal{C}_{38}^{\alpha,\beta,\gamma}\bigr)=7, \quad
\dim \mathcal{O}\bigl(\mathcal{C}_{39}^{\alpha,\beta}\bigr)=
\dim \mathcal{O}\bigl(\mathcal{C}_{09}^{\alpha,\beta}\bigr)=
\dim \mathcal{O}\bigl(\mathcal{C}_{22}^{\alpha,\beta}\bigr)=
\dim \mathcal{O}\bigl(\mathcal{C}_{31}^{\alpha,\beta,\gamma}\bigr)=6, \\
\dim \mathcal{O}(\mathcal{C}_{33})=4.
\end{gather*}

Thanks to Theorem \ref{geo1}, we have that
\smash{$\mathcal{C}_{38}^{\alpha,\beta,\gamma}$} and \smash{$\mathcal{C}_{39}^{\alpha,\beta}$} give irreducible components in the variety of $2$-dimensional compatible commutative associative algebras.
Algebras \smash{$\mathcal{C}_{09}^{\alpha,\beta}$},
\smash{$\mathcal{C}_{22}^{\alpha,\beta}$}, and~\smash{$\mathcal{C}_{31}^{\alpha,\beta, \gamma}$} have a one-dimensional subalgebra with zero multiplication (concerning both multiplications), but $\mathcal{C}_{33}$ has not.
Hence, $\mathcal{C}_{33}$ is not in the orbit closure of \smash{$\mathcal{C}_{09}^{\alpha,\beta}$},
\smash{$\mathcal{C}_{22}^{\alpha,\beta}$}, and $\mathcal{C}_{31}.$
The rest of the necessary degenerations are given by
\begin{gather*}
 \mathcal{C}_{31}^{t, \frac{t+2}{2t}, -\frac{1}{2}} \xrightarrow{ ( te_1+\frac{2 t}{3}e_2, \frac{t^2}{2}e_1+t^2e_2)} \mathcal{C}_{01} , \quad
 \mathcal{C}_{31}^{0, 0, t^{-1}} \xrightarrow{ ( t^2e_1+e_2, te_1-t^2e_2)} \mathcal{C}_{04} , \\
 \mathcal{C}_{31}^{0, t^{-1}, \alpha t^{-1}} \xrightarrow{ ( te_1-\alpha^{-1}t^2e_2, te_2)} \mathcal{C}_{06}^{\alpha\neq0} , \quad
 \mathcal{C}_{09}^{\frac{(2\alpha+t)\beta}{2\beta+t},-\frac{\beta t}{2\beta+t}} \xrightarrow{ ( \frac{2\beta +t}{2\beta}e_1-\frac{(2\beta +t)^2}{2\beta t}e_2, \frac{(2\beta+t)t}{4\beta^2}te_1)} \mathcal{C}_{10}^{\alpha,\beta\neq0} ,\\
 \mathcal{C}_{31}^{1, \alpha t^{-1}, t^{-1}} \xrightarrow{ ( te_1-t^2e_2, -t^3e_2)} \mathcal{C}_{13}^{\alpha} , \quad
 \mathcal{C}_{31}^{1, t^{-1}, 0} \xrightarrow{ (te_1+\alpha^{-1} t e_2, \alpha^{-1}t^2 e_2)} \mathcal{C}_{14}^{\alpha\neq0} , \\
 \mathcal{C}_{31}^{t^{-1}-t, 0, -1+\alpha t^{-1}} \xrightarrow{ (t^2e_1+\frac{t^2}{1-\alpha t} e_2, t e_1+\frac{t^3}{1-\alpha t} e_2)} \mathcal{C}_{20}^{\alpha,0} ,\quad
 \mathcal{C}_{31}^{t^{-1}-t, 0, -\alpha t^{-1}} \xrightarrow{ (t^2e_1-\alpha^{-1} t e_2, t e_1- \alpha^{-1} t e_2)} \mathcal{C}_{21}^{\alpha,0} , \\
 \mathcal{C}_{31}^{t^{-1}-t, \alpha t^{-1}, \beta t^{-1}} \xrightarrow{ (t^2e_1-\beta^{-1} t e_2, t e_1- \beta^{-1} t^2 e_2)} \mathcal{C}_{23}^{\alpha,\beta} , \quad
 \mathcal{C}_{31}^{\alpha, \beta, \gamma} \xrightarrow{ ( e_1, t^{-1}e_2)} \mathcal{C}_{32}^{\alpha, \beta, \gamma}. \tag*{\qed}
\end{gather*}\renewcommand{\qed}{}
\end{proof}

\subsection{The geometric classification of compatible pre-Lie algebras}

The main result of the present section is the following theorem.

\begin{Theorem}\label{geo4}
The variety of complex $2$-dimensional compatible pre-Lie algebras has
dimension $7$ and it has $14$ irreducible components defined by
\begin{gather*}
 \overline{\mathcal{O}( \mathcal{C}_{28})} , \quad
 \overline{\mathcal{O}( \mathcal{C}_{33})} , \quad
 \overline{ \mathcal{O}(\mathcal{C}_{37}^\alpha)} ,
\quad \overline{ \mathcal{O}\bigl(\mathcal{C}_{09}^{\alpha,\beta}\bigr)} ,
\quad \overline{ \mathcal{O}\bigl(\mathcal{C}_{11}^{\alpha,\beta}\bigr)} ,
\quad \overline{ \mathcal{O}\bigl(\mathcal{C}_{22}^{\alpha,\beta}\bigr)} ,
\quad\overline{ \mathcal{O}\bigl(\mathcal{C}_{27}^{\alpha,\beta}\bigr)} ,
\\ \overline{ \mathcal{O}\bigl(\mathcal{C}_{36}^{\alpha,\beta}\bigr)} ,
\quad \overline{ \mathcal{O}\bigl(\mathcal{C}_{39}^{\alpha,\beta}\bigr)} ,
\quad \overline{ \mathcal{O}\bigl(\mathcal{C}_{41}^{\alpha,\beta}\bigr)} ,
\quad \overline{ \mathcal{O}\bigl(\mathcal{C}_{24}^{\alpha,\beta,\gamma}\bigr)} ,
\quad \overline{ \mathcal{O}\bigl(\mathcal{C}_{31}^{\alpha,\beta,\gamma}\bigr)} ,
\\ \overline{\mathcal{O}\bigl(\mathcal{C}_{38}^{\alpha,\beta,\gamma}\bigr)} \quad \text{and}
\quad \overline{ \mathcal{O}\bigl(\mathcal{C}_{40}^{\alpha,\beta,\gamma}\bigr)}.
 \end{gather*}
 In particular, there are two rigid algebras in this variety.

\end{Theorem}

\begin{proof}
 Thanks to Theorem \ref{compre2}, we have the algebraic classification of
 $2$-dimensional compatible pre-Lie algebras.
After carefully checking the dimensions of orbit closures of the more important for us algebras, we have
\begin{gather*}
 \dim \mathcal{O}\bigl(\mathcal{C}_{24}^{\alpha,\beta,\gamma}\bigr)=\dim \mathcal{O}\bigl(\mathcal{C}_{31}^{\alpha,\beta,\gamma}\bigr)=\dim \mathcal{O}\bigl(\mathcal{C}_{38}^{\alpha,\beta,\gamma}\bigr)=\dim \mathcal{O}\bigl(\mathcal{C}_{40}^{\alpha,\beta,\gamma}\bigr) = 7 , \\
\dim \mathcal{O}\bigl(\mathcal{C}_{09}^{\alpha,\beta}\bigr)=
\dim \mathcal{O}\bigl(\mathcal{C}_{11}^{\alpha,\beta}\bigr)=
\dim \mathcal{O}\bigl(\mathcal{C}_{22}^{\alpha,\beta}\bigr)=
\dim \mathcal{O}\bigl(\mathcal{C}_{27}^{\alpha,\beta}\bigr) =
 \dim \mathcal{O}\bigl(\mathcal{C}_{36}^{\alpha,\beta}\bigr)=
\dim \mathcal{O}\bigl(\mathcal{C}_{39}^{\alpha,\beta}\bigr) \\
\phantom{\dim \mathcal{O}\bigl(\mathcal{C}_{09}^{\alpha,\beta}\bigr)}{}=\dim \mathcal{O}\bigl(\mathcal{C}_{41}^{\alpha,\beta}\bigr) = 6 , \\
 \dim \mathcal{O}\bigl(\mathcal{C}_{37}^\alpha\bigr) = 5, \quad
 \dim \mathcal{O}(\mathcal{C}_{28})=\dim \mathcal{O}(\mathcal{C}_{33}) = 4.
\end{gather*}

All necessary reasons for non-degenerations are listed as follows:
\begin{center}\renewcommand{\arraystretch}{1.3}
\begin{tabular}{l|l}
\hline
 \multicolumn{2}{c}{Non-degenerations reasons} \\
\hline
$\mathcal{C}_{24}^{*}
\not \rightarrow
\mathcal{C}_{22}^{\alpha,\beta}, \,
\mathcal{C}_{26}^{\alpha}, \,
\mathcal{C}_{27}^{\alpha,\beta}, \,
\mathcal{C}_{28}$
&
$\mathcal R=\bigl\{c_{21}'^1=c_{21}'^2=c_{22}^1=c_{22}^2=c_{22}'^1=c_{22}'^2=0\bigr\}$\\
\hline
$
\mathcal{C}_{31}^{*}
\not \rightarrow
\mathcal{C}_{33}, \,
\mathcal{C}_{36}^{\alpha,\beta}, \,
\mathcal{C}_{37}^{\alpha}$
&
$\mathcal R=\bigl\{
c_{22}'^1=c_{23}'^2=c_{22}^1=c_{22}^2=0
\bigr\}
$\\
\hline
$\mathcal{C}_{36}^{*}
\not \rightarrow \mathcal{C}_{37}^{\alpha}$
&
$\mathcal R=\bigl\{
c_{22}^1=c_{12}^1=c_{12}'^1=0\bigr\}
$\\
\hline
$\mathcal{C}_{40}^{*}\not \rightarrow\mathcal{C}_{41}^{\alpha,\beta}$
&
$\mathcal R=\bigl\{
c_{22}^1=c_{12}^1=c_{12}'^1=0, \ 2 c_{11}'^1=c_{12}'^1
\bigr\}$\\
\hline
\end{tabular}
\end{center}

The rest of the non-degenerations between indicated algebras follows from Remark \ref{remrem}.

Thanks to Theorem \ref{geo1}, we have that
$ \mathcal{C}_{33}$,
$ \mathcal{C}_{09}^{\alpha,\beta}$,
$ \mathcal{C}_{22}^{\alpha,\beta}$,
$ \mathcal{C}_{39}^{\alpha,\beta}$,
$ \mathcal{C}_{31}^{\alpha,\beta, \gamma}$, and
$ \mathcal{C}_{38}^{\alpha,\beta,\gamma}$
 give irreducible components in the variety of $2$-dimensional compatible Novikov algebras.
 Hence, each $2$-dimensional compatible Novikov algebras is on the orbit closure of these algebras. The rest of the necessary degenerations are given by
\begin{gather*}
 \mathcal{C}_{24}^{ 1,1+ t^{-1}, -1+ t^{-1}} \xrightarrow{ (te_1 +(-1+t) e_2, t^{2}e_1+t e_2 )} \mathcal{C}_{02} ,\quad
 \mathcal{C}_{24}^{ 1, t^{-1}, \alpha t^{-1}} \xrightarrow{ (te_1 +\frac{t^2}{1-\alpha} e_2, t^{2}e_1+ e_2 )} \mathcal{C}_{05}^{\alpha\neq 1} ,\\
 \mathcal{C}_{40}^{t^{-1},0, \frac{1}{2t}} \xrightarrow{ (te_1, te_2)} \mathcal{C}_{08} ,\quad
 \mathcal{C}_{11}^{\alpha,\alpha+t, } \xrightarrow{ (e_1+\beta t^{-1} e_2, e_2)} \mathcal{C}_{12}^{\alpha,\beta} ,\\
 \mathcal{C}_{24}^{ -1+\alpha, t^{-1}, \alpha t^{-1}} \xrightarrow{ (te_1 -\frac{t^2}{\alpha-1} e_2, t^{2}e_1-t^3 e_2 )} \mathcal{C}_{16}^{\alpha\neq 1} ,\\
 \mathcal{C}_{24}^{ -2, \alpha+t^{-1}, -\alpha +t^{-1}} \xrightarrow{ (te_1 -\alpha^{-1}t e_2, t^{2}e_1+2\alpha^{-1}t^2 e_2 )} \mathcal{C}_{17}^{\alpha\neq 0} ,\\
 \mathcal{C}_{40}^{ -\alpha -2t^{-2}, -6\alpha -6t^{-2}, -\frac{4+\alpha t^2}{6t^2}} \xrightarrow{ (te_1-3 te_2, \frac{t^2}{2}e_1-3 t^2e_2 )} \mathcal{C}_{19}^{\alpha} ,\\
 \mathcal{C}_{24}^{ -1-t^{-2}, -(\beta+t)t^{-2}, (t-\alpha) t^{-2}} \xrightarrow{ (-t^3e_1 +\frac{t^4+t^6}{1+\alpha t -\beta t} e_2, -t^{2}e_1+ \frac{t^5}{\beta t -\alpha t -1}e_2 )} \mathcal{C}_{20}^{\alpha,\beta} ,\\
 \mathcal{C}_{24}^{ -1-t^{-2}, -\beta t^{-2}, -\alpha t^{-2}} \xrightarrow{ (-t^3e_1 +\frac{t^3+t^5}{\alpha -\beta } e_2, -t^{2}e_1- \frac{t^4}{ \alpha -\beta}e_2 )} \mathcal{C}_{21}^{\alpha,\beta} ,\quad
 \mathcal{C}_{27}^{\alpha -5t^2,t^{-2}} \xrightarrow{ (e_1 + 2t^2 e_2, t e_2 )} \mathcal{C}_{26}^{\alpha} ,\\
 \mathcal{C}_{24}^{ \alpha, \beta, \gamma} \xrightarrow{ (e_1, t^{-1}e_2 )} \mathcal{C}_{25}^{\alpha, \beta, \gamma}. \tag*{\qed}
\end{gather*} \renewcommand{\qed}{}
\end{proof}

\subsection*{Acknowledgements}

The authors thank Amir Fern\'andez Ouaridi for sharing some useful software programs.
The work is supported by
FCT UIDB/00212/2020, UIDP/00212/2020 and 2022.14950.CBM (Programa PESSOA).

\pdfbookmark[1]{References}{ref}
\LastPageEnding

\end{document}